\documentclass[12pt,reqno,twoside]{amsart}

\usepackage{amssymb}
\usepackage{amsmath}
\usepackage[colorlinks=true,linkcolor=blue,citecolor=blue]{hyperref}
\usepackage[latin1]{inputenc}
\usepackage{color}
\usepackage{graphicx}

\makeatletter
\newcommand*\dashline{\rotatebox[origin=c]{90}{$\dabar@\dabar@\dabar@\dabar@$}}
\makeatother

\usepackage[all]{xy}
\usepackage{pstricks}
\newcommand{\mm}{\mathfrak m}

\newcommand{\lsf}{\mathsf{l}}

\newcommand{\nsf}{\mathsf{n}}
\newcommand{\rsf}{\mathsf{r}}
\newcommand{\tsf}{\mathsf{t}}
\newcommand{\Vsf}{\mathsf{V}}
\newcommand{\vsf}{\mathsf{v}}
\newcommand{\xsf}{\mathsf{x}}

\newcommand{\Z}{\mathbb{Z}}

\newcommand{\N}{\mathbb{N}}
\newcommand{\Pbb}{\mathbb{P}}

\newcommand{\Fc}{\mathcal{F}}

\newcommand{\Jc}{\mathcal{J}}

\newcommand{\Ncc}{\mathcal{N}}

\newcommand{\Sc}{\mathcal{S}}

\newcommand\bsa{{\boldsymbol a}}
\newcommand\bsb{{\boldsymbol b}}
\newcommand\bsc{{\boldsymbol c}}

\newcommand\bsx{{\boldsymbol x}}
\newcommand\bsy{{\boldsymbol y}}
\newcommand\bsz{{\boldsymbol z}}

\DeclareMathOperator{\pnt}{\raise 0.5mm \hbox{\large\bf.}}

\DeclareMathOperator{\GL}{GL}

\DeclareMathOperator{\chara}{char}
\DeclareMathOperator{\Tor}{Tor}

\DeclareMathOperator{\rank}{rank}

\DeclareMathOperator{\reg}{reg}

\def\+#1{\relax\ifmmode\if\noexpand #1\relax \mathop{\kern
    0pt^+{#1}}\nolimits\else \kern 0pt^+\!#1 \fi\else$^*$#1\fi}

\let\phi=\varphi
\let\:=\colon

\numberwithin{equation}{section}

\newtheorem{thm}{\bf Theorem}[section]
\newtheorem{lem}[thm]{\bf Lemma}

\newtheorem{prop}[thm]{\bf Proposition}

\theoremstyle{definition}
\newtheorem{defn}[thm]{\bf Definition}

\theoremstyle{plain}
\newtheorem*{thm*}{Theorem}
\newtheorem*{lem*}{Lemma}
\newtheorem*{cor*}{Corollary}
\newtheorem*{claim*}{Claim}
\newtheorem*{defn*}{Definition}

\theoremstyle{remark}
\newtheorem{rem}[thm]{Remark}
\newtheorem{constr}[thm]{Construction}
\newtheorem{ex}[thm]{Example}

\title[Matrices of linear forms]{Koszul determinantal rings and $2\times e$ matrices of linear forms}
\author[H. D. Nguyen]{Hop D. Nguyen}
\address{Dipartimento di Matematica, Universit\`a di Genova, Via Dodecaneso 35, 16146 Genoa, Italy}
\address{Ernst-Abbe-Platz 5, Appartment 605, 07743 Jena, Germany}
\date{\today}
\email{ngdhop@gmail.com}
\author[P. D. Thieu]{Phong Dinh Thieu}
\address{Universit\"at Osnabr\"uck, Institut f\"ur Mathematik, 49069 Osnabr\"uck, Germany}
\address{Department of Mathematics, Vinh University, 182 Le Duan, Vinh City, Vietnam} 
\email{thieudinhphong@gmail.com}
\author[T. Vu]{Thanh Vu}
\address{Department of Mathematics, University of California at Berkeley, Berkeley CA 94720}
\email{vqthanh@math.berkeley.edu}
\thanks{The first named author is grateful to the support of the Vigoni project (in 2011) and the CARIGE foundation.}
\subjclass[2010]{13D02, 13C40}
\keywords{Koszul algebras, determinantal ring, rational normal scrolls, Kronecker-Weierstrass normal form.}
\begin{document}

\begin{abstract}
Let $k$ be an algebraically closed field of characteristic $0$. Let $X$ be a $2\times e$ matrix of linear forms over a polynomial ring $k[\mathsf{x}_1, \ldots,\mathsf{x}_n]$ (where $e,n\ge 1$). We prove that the determinantal ring $R = k[\mathsf{x}_1,\ldots,\mathsf{x}_n]/I_2(X)$ is Koszul if and only if in any Kronecker-Weierstrass normal form of $X$, the largest length of a nilpotent block is at most twice the smallest length of a scroll block. As an application, we classify rational normal scrolls whose all section rings by natural coordinates are Koszul. This result settles a conjecture due to Conca.
\end{abstract}

\maketitle
\section{Introduction}
\label{intro}
Let $k$ be an algebraically closed field of characteristic $0$, $R$ a commutative, standard graded $k$-algebra. The last condition means that $R$ is $\Z$-graded, $R_0=k$ and $R$ is generated as a $k$-algebra by finitely many elements of degree $1$. We say that $R$ is a {\em Koszul algebra} if $k$ has linear resolution as an $R$-module. Denote by $\reg_R M$ the Castelnuovo-Mumford regularity of a finitely generated graded $R$-module $M$. An equivalent way to express the Koszulness of $R$ is the condition $\reg_R k=0$. Effective techniques to prove Koszulness include Gr\"obner deformation, Koszul filtrations, computation of the Betti numbers of $k$ for toric rings, among others. For some survey articles on Koszul algebras, we refer to \cite{CNR}, \cite{Fr2}. 

In this paper, we study the Koszul property of linear sections of rational normal scrolls. By abuse of terminology, we use ``rational normal scrolls" to refer to the homogeneous coordinate rings of the corresponding varieties. These graded algebras are defined by the ideals of 2-minors of some $2\times e$ matrices of linear forms, where $e\ge 1$. The homogeneous coordinate rings of the Segre embedding $\Pbb^1\times \Pbb^e\to \Pbb^{2e+1}$ and the Veronese embedding $\Pbb^1\to \Pbb^e$ are among the examples; in fact they are special instances of rational normal scrolls. The rational normal scrolls are a classical and widely studied class of varieties with minimal multiplicity, whose classification is known from works of Del Pezzo and Bertini; see \cite{EH}. 

Let $X$ be a $2\times e$ matrix of linear forms over a polynomial ring $S = k[\xsf_1, \ldots,\xsf_n]$. Let $R = k[\xsf_1,\ldots,\xsf_n]/I_2(X)$ be the determinantal ring of $X$. Algebraic properties of such determinantal rings $R$ were studied in the literature, see \cite{Ch}, \cite{Ca} and \cite{ZZ}. The Kronecker-Weierstrass theory of matrix pencils (see Section \ref{background}) played an important role in these works.

Concerning the Koszul property, any rational normal scroll is Koszul since it has regularity $1$. In fact, any rational normal scroll is also G-quadratic, namely its defining ideal has a quadratic Gr\"obner basis with respect to a suitable term order; see \cite{ZZ} for a generalization. In this paper, we are able to classify Koszul determinantal rings of $2\times e$ matrices of linear forms using the Kronecker-Weierstrass theory. The main technical result of the paper is:
\begin{thm}\label{main}
Let $X$ be a $2\times e$ matrix of linear forms (where $e\ge 1$) and $R=k[X]/I_2(X)$ the determinantal ring of $X$. Then $R$ is Koszul if and only if $m\le 2n$, where $m$ is length of the longest nilpotent block and $n$ is length of the shortest scroll block in any Kronecker-Weierstrass normal form of $X$. (The last condition holds if there is either no such nilpotent block or no such scroll block.)
\end{thm}
Since $k$ is algebraically closed and $\chara k=0$, we may assume that $X$ is already in the Kronecker-Weierstrass normal form. Denote $m$ the length of the longest nilpotent block and $n$ the length of the shortest scroll block of $X$. We deduce the sufficient condition in Theorem \ref{main} by constructing a Koszul filtration for $R$ given that $X$ satisfying the {\em length condition} $m \le 2n$ (Construction \ref{constr_filtration}). The construction supplies new information even for rational normal scrolls. 

As applications, we are able to characterize the rational normal scrolls which ``behave like" algebras defined by quadratic monomial ideals. Let us introduce some more notation. Let $S=k[\xsf_1,\ldots,\xsf_n]$ be a standard graded polynomial algebra which surjects onto the $k$-algebra $R$ (not necessarily a determinantal ring). For any finitely generated graded $R$-module $M$, we use $\reg M$ to denote $\reg_S M$, which is an invariant of $M$. Koszul algebras defined by quadratic monomial relations (see Fr\"oberg's paper \cite{Fr1}) have very strong resolution-theoretic properties. If $R=S/I$ where $I$ is a quadratic monomial ideal of $S$, for any set of variables $Y\subseteq \{\xsf_1,\ldots,\xsf_n\}$ of $S$, we have:
\begin{enumerate}
\item $\reg_R R/(Y ) \le \reg R$;
\item  $R/(Y)$ is a Koszul algebra;
\item (see \cite{HHR}) $\reg_R R/(Y )= 0$.
\end{enumerate} 
Thus all the linear sections by natural coordinates of $R$ have a linear resolution over $R$ and are Koszul algebras. In fact, (i) and (ii) are consequences of (iii) by Lemma \ref{lem_regularity} below.

For $R$ being a rational normal scroll of type $(\nsf_1,\ldots,\nsf_{\tsf})$ where $\tsf\ge 1, 1\le \nsf_1 \le \cdots \le \nsf_{\tsf}$, $R$ is defined by the ideal of maximal minors of the matrix
\[
\left(\begin{matrix}
       y_{1,1}   & y_{1,2} & \ldots  & y_{1,\nsf_1} \\
       y_{1,2}   & y_{1,3} & \ldots  & y_{1,\nsf_1+1} \end{matrix}\right. \left. \dashline ~ \begin{matrix}
                                       y_{2,1} & y_{2,2} &\ldots & y_{2,\nsf_2}  \\
                                       y_{2,2} & y_{2,3} &\ldots & y_{2,\nsf_2+1}
                                       \end{matrix}\right. \left. \dashline \cdots  \right. 
                                                                              \left. \dashline ~ \begin{matrix}
                                                                              y_{\tsf,1} & y_{\tsf,2} &\ldots & y_{\tsf,\nsf_{\tsf}}  \\
                                                                              y_{\tsf,2} & y_{\tsf,3} &\ldots & y_{\tsf,\nsf_{\tsf}+1}
                                                                              \end{matrix} \right),
\]
where $y_{1,1},y_{1,2}\ldots,y_{1,\nsf_1+1},y_{2,1},\ldots,y_{\tsf,\nsf_{\tsf}+1}$ are distinct variables. By {\em the set of natural coordinates} of $R$, we mean $\{y_{1,1},y_{1,2}\ldots,y_{1,\nsf_1+1},y_{2,1},\ldots,y_{\tsf,\nsf_{\tsf}+1}\}$. The main application of Theorem \ref{main} is:
\newpage
\begin{thm}\label{linearly_Koszul}Let $R$ be a rational normal scroll of type $(\nsf_1,\ldots,\nsf_{\tsf})$ where $1\le \nsf_1 \le \cdots \le \nsf_{\tsf}$. Let $Y$ be a subset of the set of natural coordinates of $R$. 
\begin{enumerate}
\item $\reg R/(Y) \le \reg R$ for every possible choice of $Y$ if and only if $R$ is balanced, i.e., $\nsf_{\tsf} \le \nsf_1+1$.
\item $R/(Y)$ is a Koszul algebra for every possible choice of $Y$ if and only if $\nsf_{\tsf}\le 2\nsf_1$.
\end{enumerate}
\end{thm}
Note that under the same assumptions, we also have

(iii) (Conca \cite{Con1}) $\reg_R R/(Y)=0$ for every possible choice of $Y$ if and only if $\nsf_{\tsf}=\nsf_1$. Moreover, in that case, $R$ is strongly Koszul in the sense of \cite{HHR}.

The last result was mentioned by Conca in \cite{Con1} without proof; we give an argument here. Part (i) is proved by using a formula of Castelnuovo-Mumford regularity of linear sections of $R$ by Catalano-Johnson \cite{Ca} and Zaare-Nahandi and Zaare-Nahandi \cite{ZZ}. This was conjectured in \cite{Con1}. Part (ii) confirms a conjecture proposed by Conca \cite{Con1}, which was made based on numerical evidences. Note that arguing a little bit further, we do not have to put any restriction on $k$ in Theorem \ref{linearly_Koszul}; see Remark \ref{rem_quotient}. Studying Conca's conjecture was the original motivation of this project.

The paper is structured as follow. In Section \ref{background} we recall Kronecker-Weierstrass theory of matrix pencils, results about determinantal rings of \cite{Ca}, \cite{Ch}, \cite{ZZ} and the notion of Koszul filtration \cite{CTV}. In Section \ref{sect_KW_quotient}, particularly in Proposition \ref{KW_form} and Lemma \ref{KW_quotient}, we describe the changes in the Kronecker-Weierstrass normal forms after going modulo certain linear forms. Section \ref{Koszul_sufficient} is devoted to the proof of the sufficiency part in Theorem \ref{main} using Koszul filtration (Construction \ref{constr_filtration}). To verify the validity of our Koszul filtration, we use the Hilbert series formula of $2\times e$ matrices of linear forms discovered by Chun and a Gr\"obner basis formula for such matrices due to Rahim Zaare-Nahandi and Rashid Zaare-Nahandi. In Section \ref{Koszul_necessary}, the necessity part in Theorem \ref{main} is established by using the monoid presentation of a rational normal scroll and a formula for multigraded Betti numbers of $k$ due to Herzog, Reiner and Welker \cite{HRW}. We prove Theorem \ref{linearly_Koszul} in Section \ref{sect_application}. As another application of Theorem \ref{main}, we classify completely the rational normal scrolls whose all quotients by linear ideals are Koszul algebras (Theorem \ref{ul_Koszul_scrolls}).

\section{Background}
\label{background}
\subsection{Kronecker-Weierstrass normal forms}
Let $k$ be an algebraically closed field of characteristic zero. We review the theory of Kronecker-Weierstrass normal forms in this section. For a detailed discussion, we refer to \cite[Chapter XII]{Ga}. For more recent treatment and algorithms for finding the Kronecker-Weierstrass normal forms, we refer to \cite{BV}, \cite{V}. Let $S=k[\xsf_1,\ldots,\xsf_{n}]$ be a polynomial ring over a field $k$ (where $n\ge 1$). Let $\xsf_1^*,\ldots,\xsf_{n}^*$ be the basis for the dual vector space of the $k$-vector space $\Vsf$ with basis $\xsf_1,\ldots,\xsf_{n}$. Let $X$ be a $2\times e$ matrix of linear forms in $S$ (where to avoid triviality, we assume $e\ge 2$).

Each row of $X$ can be identified with a matrix in $M_{e\times n}$ in the following way: let $\rsf=(\lsf_1,\ldots,\lsf_e)$ be a row, then for $i=1,\ldots,n$, the $i$th {\em column} of the matrix $M_{\rsf}$ is given by $(\xsf_i^*(\lsf_1),\ldots,\xsf_i^*(\lsf_e))^T$. Thus
\[
M_{\rsf}=\left(\begin{matrix}
\xsf^*_1(\lsf_1) & \xsf^*_2(\lsf_1) &\ldots & \xsf^*_n(\lsf_1)\\
\xsf^*_1(\lsf_2) & \xsf^*_2(\lsf_2) &\ldots & \xsf^*_n(\lsf_2)\\
\ldots           &\ldots            &\ldots &\ldots\\
\xsf^*_1(\lsf_e) & \xsf^*_2(\lsf_e) &\ldots & \xsf^*_n(\lsf_e)
\end{matrix}\right) \in M_{e\times n}.
\]

Now $X$ can be identified with the vector subspace of $\Vsf^e$ generated by two rows $\rsf_1,\rsf_2$ of $X$. In turn, this vector subspace of $\Vsf^e$ can be identified with the vector subspace $\Vsf_X$ generated by two matrices $M_{\rsf_1},M_{\rsf_2}$ of $M_{e\times n}$.

If $\dim \Vsf_X \le 1$ then $\rsf_1,\rsf_2$ are linearly dependent and $I_2(X)=0$. So let us assume that $\dim \Vsf_X = 2$. From the Kronecker-Weierstrass theory of matrix pencils, there exist invertible matrices $C\in \GL (k^e), C'\in \GL(\Vsf)$ such that 
\[
C(M_{\rsf_1}+v M_{\rsf_2})C' =
\left( \begin{matrix}
L_{m_1-1}^T &  & & & & & & & \\
& \ddots     &  & & & & & &  \\
& & L_{m_c-1}^T & & & & & &  \\
& & & L_{n_1} & & & & & \\
& & & & \ddots & & & & & \\ 
& & & & & L_{n_d} & & &  \\
& & & & & & J_{p_1, \lambda_1} & &  \\
& & & & & & & \ddots & & \\
& & & & & & & & J_{p_g, \lambda_g}  
\end{matrix} \right),
\]
where $v$ is a variable, 
\[
L_{m-1}=
\left(\begin{matrix}
1 & v & \cdots & 0 & 0\\
0 & 1 &  v     & \cdots & 0\\
\vdots & \vdots & \ddots & \ddots & \vdots\\
0 & 0 & \cdots & 1      & v
\end{matrix}\right) \in M_{(m-1)  \times m},
\]
and 
\[
J_{p, \lambda}=
\left(\begin{matrix}
\lambda v+1 & v & \cdots & 0 &0 \\
0   & \lambda v+1 & v & \cdots & 0\\
\vdots & \vdots & \ddots   & \ddots & \vdots\\
0 & 0  & \cdots & \lambda v+1 & v \\
0 & 0  & \cdots & 0 & \lambda v+1
\end{matrix} \right) \in M_{p \times p}.
\]
Since $C$ and $C'$ are invertible, $X$ defines the same determinantal ideal as the matrix with rows corresponding to the matrices $CM_{\rsf_1} C', C M_{\rsf_2}C'$. Concretely, the last matrix is a concatenation of the following three types of matrices
\[
\left(\begin{matrix}
       x_{i,1}   & x_{i,2} & \ldots  & x_{i,m_i-1}   & 0\\
       0         & x_{i,1} & \ldots  & x_{i,m_i-2} & x_{i,m_i-1}\end{matrix}\right),
\]
\newpage
\[
\left(\begin{matrix}
       y_{j,1}   & y_{j,2} & \ldots  & y_{j,n_j} \\
       y_{j,2}   & y_{j,3} & \ldots  & y_{j,n_j+1} \end{matrix}\right),
\]
and
\[
\left(\begin{matrix}
z_{l,1}                   & z_{l,2}       & \ldots  & z_{l,p_l-1}        & z_{l,p_l} \\
z_{l,2}+\lambda_l z_{l,1}   & z_{l,3}+\lambda_l z_{l,2} & \ldots  & z_{l,p_l}+\lambda_l z_{l,p_l-1} &\lambda_l z_{l,p_l}\end{matrix}\right),
\]
where $\bsx,\bsy,\bsz$ are independent linear forms of $S$, $1\le i\le c, 1\le j\le d$ and $1\le l\le g$ for some $c,d,g\ge 0$. We call these matrices  nilpotent block, scroll block and Jordan block with eigenvalue $\lambda_l$, respectively. By definition, the {\em length} of these blocks are $m_i,n_j$ and $p_l$, respectively. The numbers $c,d$ and the lengths of nilpotent and scroll blocks $m_i,n_j$ where $1\le i\le c, 1\le j \le d$ are invariants of $X$ but $\lambda_l$s are not. See \cite{Ga}, \cite[Section 3]{Ca} for more details. 

For the convenience of our arguments, we write the columns of nilpotent blocks with the reverse order and re-index. Hence in our notation, nilpotent blocks are of the form
\[
\left(\begin{matrix}
       0        & x_{i,1}    & x_{i,2}     & \ldots          & x_{i,m_i-2} & x_{i,m_i-1}\\
      x_{i,1}   & x_{i,2}    & x_{i,3}     & \ldots          & x_{i,m_i-1}   & 0 \end{matrix}\right).
\]
We call concatenation of the above scroll blocks, nilpotent blocks (in our notation) and Jordan blocks obtained from $CM_{\rsf_1} C'$ and $CM_{\rsf_2} C'$ a {\em Kronecker-Weierstrass normal form} of $X$. 

Fix a Kronecker-Weierstrass normal form of $X$. For our purpose, Jordan blocks with different eigenvalues behave differently, so we will refine our notation. We assume that the Jordan blocks of $X$ are divided into $g_i$ Jordan blocks with eigenvalue $\lambda_i$, for $i = 1, \ldots, t$. Here, the eigenvalues $\lambda_1, \lambda_2, \ldots, \lambda_t$ are pairwise distinct. Concretely, 
\[
X=\left(\begin{matrix} X_{\text{nil}} \dashline X_{\text{sc}} \dashline X^{1}_1 & X^1_2 & \cdots & X^1_{g_1} \dashline \cdots \dashline X^t_1 & X^t_2 & \cdots X^t_{g_t}                      \end{matrix}\right),
\]                  
where 
\[
X^i_j=\left(\begin{matrix} z^i_{j,1}     & z^i_{j,2}  &\ldots & z^i_{j,p_{ij}} \\ 
                       z^i_{j,2} + \lambda_i z^i_{j,1}     & z^i_{j,3} + \lambda_i z^i_{j,2}  &\ldots &\lambda_i z^i_{j,p_{ij}}
                                       \end{matrix}\right).
\]
Here $X_{\text{nil}}, X_{\text{sc}}$ denote the submatrices of $X$ consisting of nilpotent blocks and scroll blocks, respectively. In addition, we assume that $p_{i1} \ge p_{i2} \ge \cdots \ge p_{ig_i}$ for $1\le i\le t$. 

We call the sequence $(m_1 \le m_2 \le \cdots \le m_c, n_1 \le n_2 \le \cdots \le n_d, p_{11} \ge \cdots \ge p_{1g_1},\ldots,p_{t1}\ge \cdots \ge p_{tg_t})$ the {\em length sequence} of $X$. We write the length sequence of (the given Kronecker-Weierstrass normal form of) $X$ as follow
\[
(\underbrace{m_1,\ldots,m_c}_{\Ncc},\underbrace{n_1,\ldots,n_d}_{\Sc},\underbrace{p_{11},\ldots,p_{1g_1},p_{21},\ldots,p_{tg_t}}_{\Jc}).
\]

\begin{ex}
\label{ex_KW}
Let $R$ be the $(2,4)$ scroll defined by the following matrix
\[
\left(\begin{matrix}
        y_{11} & y_{12} \\ 
        y_{12} & y_{13}  \end{matrix}\right. 
        \dashline \left. ~ \begin{matrix}
                                      y_{21} & y_{22} & y_{23} & y_{24} \\ 
                                      y_{22} & y_{23} & y_{24} & y_{25}
                                   \end{matrix} \right),
\]
We show that $R/(y_{23})$ is defined by two Jordan blocks with eigenvalue $0$ and $1$ and a scroll block of length $2$.

Changing variables for simplicity, clearly $R/(y_{23})$ is defined by the following matrix
\[
\left(\begin{matrix}
           z_1 & z_2\\ 
           z_2 & z_3
          \end{matrix} \right. ~ \left. \begin{matrix}
        t_1 & t_2 & 0   & u_1 \\ 
        t_2 & 0   & u_1 & u_2 \end{matrix}\right). 
\]
Adding the second row to the first row, we get
\[
\left(\begin{matrix}
          z_1+z_2 & z_2+z_3\\ 
           z_2     & z_3
          \end{matrix} \right.  ~ \left. \begin{matrix}
        t_1+t_2 & t_2 & u_1   & u_1+u_2 \\ 
        t_2     & 0   & u_1   & u_2 \end{matrix}\right).
\]
Multiplying the last column with -1, then swapping its to the previous column, we get
\[
\left(\begin{matrix}
                                      z_1+z_2 & z_2+z_3\\ 
                                      z_2     & z_3
                                   \end{matrix} \right. ~ \left. \begin{matrix}
        t_1+t_2 & t_2 & -u_1-u_2 & u_1   \\ 
        t_2     & 0   & -u_2    & u_1   \end{matrix}\right).
\]
Let $w_1=-u_1-u_2$, the last matrix is nothing but
\[
\left(\begin{matrix}
                                      z_1+z_2 & z_2+z_3\\ 
                                      z_2     & z_3
                                   \end{matrix} \right. \dashline  \left. \begin{matrix}
        t_1+t_2 & t_2 & w_1    & u_1   \\ 
        t_2     & 0   & u_1+w_1    & u_1   \end{matrix}\right).
\]
Adding the second column to the first, we get
\[
\left(\begin{matrix}
                                      z_1+2z_2+z_3 & z_2+z_3\\ 
                                      z_2+z_3     & z_3
                                   \end{matrix} \right. ~ \dashline  \left. \begin{matrix}
        t_1+t_2 & t_2 \\  
        t_2     & 0 
        \end{matrix} \right. ~ \dashline  \left. \begin{matrix}
        w_1    & u_1   \\
        u_1+w_1    & u_1     
        \end{matrix}\right).
\]
which is a concatenation of a scroll block, a Jordan block with eigenvalue 0 and another Jordan block with eigenvalue 1.
\end{ex}
\begin{rem}
\label{rem_quotient}
Note that using arguments similar to that of Example \ref{ex_KW}, one can show that if $R$ is a rational normal scroll and $Y$ is a set of natural coordinates then $R/(Y)$ is defined by nilpotent, scroll and Jordan blocks with eigenvalue 0 or 1. There is no need to assume that $k$ is algebraically closed of characteristic zero in these arguments. We leave the details to the interested reader.
\end{rem}

\subsection{Hilbert series and Castelnuovo-Mumford regularity}
\label{sect_CMreg}
Let $R$ be a standard graded $k$-algebra. For a finitely generated graded $R$-module $M$, we define the Castelnuovo-Mumford regularity of $M$ by
\[
\reg_R M=\sup\{j-i: \Tor^R_i(k,M)_j\neq 0\}.
\]
The following result is well-known; we state it for ease of reference.
\begin{lem}[{\cite[Proposition 2.1]{Cha}}]
\label{lem_regularity}
Let $S\to R$ be a surjection of standard graded $k$-algebras, $M$ a finitely generated graded $R$-module. Then
\begin{enumerate}
\item $\reg_S M\le \reg_S R +\reg_R M$.
\item If $\reg_S R\le 1$ then $\reg_R M \le \reg_S M$.
\end{enumerate}
\end{lem}
Let $X$ be a Kronec\-ker-Weierstrass matrix of length sequence 
\[
\underbrace{m_1, \ldots , m_c}_{\Ncc},\underbrace{n_1,\ldots,n_d}_{\Sc}, \underbrace{p_{11},\ldots, p_{1g_1},\ldots,p_{t1}, \ldots, p_{tg_t}}_{\Jc}.
\]
Denote $m=m_c = \max\{m_1,\ldots,m_c\}.$ For integers $b, q$, let $N(n_1,\ldots,n_d;b,q)$ denote the cardinality of the set
\[
\left\{(\vsf_1,\ldots,\vsf_d):\vsf_j \in \Z_{\ge 0}, \sum_{j=1}^{d}n_j\vsf_j \le b-1 ~\textnormal{and $\sum_{i=1}^{d}\vsf_i=q-1$}\right \}.
\]
We immediately have the following:
\begin{lem}
\label{lem_N_function}
If $b \le (q-1)\cdot \min\{n_1,\ldots,n_d\}$ then $N(n_1,\ldots,n_d,b,q)=0$.\qed
\end{lem}

Let $R$ be the determinantal ring of $X$. Let $R'$ be the determinantal ring of the submatrix of $X$ consisting of Jordan and scroll blocks. We cite the following result for later usage.
\begin{thm}[Chun, {\cite[2.2.3]{Ch}}]
\label{Hilbert_series}
The Hilbert series of $R=k[X]/I_2(X)$ is given by 
\[
H_R(v)=(\sum_{i=1}^{c}m_i-c)v+\sum_{q=2}^{m}\left(\sum_{i=1}^c\sum_{r=0}^{m_i-2}N(n_1,\ldots,n_d;m_i-1-r,q)\right)v^q  +H_{R'}(v).
\]
\end{thm}
The regularity of the determinantal rings of $2\times e$ matrices of linear forms can be computed as follow.
\begin{thm}[{\cite[Section 5]{Ca}},~{\cite[Theorem 4.2]{ZZ}}]
\label{regularity_of_generalized_scroll}
Let $X$ be a $2\times e$ matrix of linear forms such that $I_2(X)\neq 0$. If in a Kronecker-Weierstrass normal form of $X$, $m$ is the length of the longest nilpotent block and $n$ is the length of the shortest scroll block, then $\reg k[X]/I_2(X)=1$ if either $m \le 1$ or $n = 0$, and $\lceil \frac{m-1}{n}\rceil$ otherwise.
\end{thm}

\subsection{Koszul filtrations} 
\label{sect_K_filtration}
We recall the following notion due to Conca, Trung and Valla \cite{CTV} which is implicit in \cite{BHV}.
\begin{defn}[Koszul filtration]
\label{defn_filtration}
Let $R$ be a standard graded $k$-algebra with graded maximal ideal $\mm$. Let $\Fc$ be a set of ideals of $R$ such that
\begin{enumerate}
\item every ideal in $\Fc$ is generated by linear forms;
\item $0$ and $\mm$ belong to $\Fc$;
\item (colon condition) if $I\neq 0$ and $I\in \Fc$ then there exists an ideal $J\in \Fc$ and a linear form $x\in R_1 \setminus 0$ such that
$I=J+(x)$ and $J:I \in \Fc$.
\end{enumerate}
Then $\Fc$ is called a {\em Koszul filtration} of $R$.
\end{defn}
In the same paper, the authors proved that if such a Koszul filtration exists then $\reg_R R/I=0$ for every $I\in \Fc$. In particular, choosing $I=\mm$, $R$ is Koszul. Furthermore,  for $I\in \Fc$, the quotient ring $R/I$ is Koszul by applying Lemma \ref{lem_regularity}(ii) to $M=k$. 

\subsection{Gr\"obner bases in the absence of nilpotent blocks}
\label{sect_Groebner}
We need of the following result on Gr\"obner basis, which is crucial to our arguments in the sequel. Let $X$ be a Kronecker-Weierstrass matrix with the length sequence 
$$
(\underbrace{m_1 \le \cdots \le m_c}_{\Ncc},\underbrace{n_1 \le \cdots \le n_d}_{\Sc},\underbrace{p_{11} \ge \cdots \ge p_{1g_1},\ldots,p_{t1}\ge \cdots \ge p_{tg_t}}_{\Jc})
$$ and order the blocks of $X$ according to its length sequence. For our purpose, we have chosen a different order of blocks comparing with that of \cite[Proposition 3.1]{ZZ}. On the other hand, for the next result, the argument of {\it loc.~ cit.} carries over verbatim.
\begin{lem}[{\cite[Proposition 3.1]{ZZ}}]
\label{lem_Groebner}
Assume that $X$ has no nilpotent block. Order the variables in $k[X]$ such that they are {\em decreasing} on the first row and the last variable of a block is {\em larger} than the first variable of its adjacent block on the right. Then in the induced degree revlex order, the 2-minors of $X$ form a Gr\"obner basis for $I_2(X)$.
\end{lem}

\section{Kronecker-Weierstrass normal forms of certain section rings}
\label{sect_KW_quotient} 
Let $X$ be a $2\times e$ matrix of linear forms in a polynomial ring $S=k[\xsf_1,\ldots,\xsf_n]$ (where $e,n \ge 1$ and $I_2(X)\neq 0$). Let $A,B\in M_{e\times n}$ be the matrix corresponding to the rows of $X$ as in Section \ref{background}. Consider the matrix pencil $A + v B$, where $v$ is an indeterminate. The largest number $r$ such that there exists an $r$-minor of $A+vB$ with non-zero determinant is called the {\it rank} of $A+vB.$

By \cite[page 30, Theorem 4]{Ga} and its proof we have the following criterion for the existence scroll blocks and information about their lengths.

\begin{lem}\label{lem_scroll} Some (equivalently, every) Kronecker-Weierstrass normal form of $X$ has a scroll block if and only if $\rank (A + vB) < \min\{n,e\}$. Moreover 

(i) If some Kronecker-Weierstrass normal form of $X$ contains a scroll block of length $s\ge 1$ then there exist $(s+1)$ linearly independent vectors $w_0,w_1,\ldots,w_s$ in $k^n$ such that
\begin{equation}\label{eq_lem_scroll}
Aw_0=0,Bw_0=Aw_1,\ldots,Bw_{s-1}=Aw_s,Bw_s=0.
\end{equation}

(ii) Assume that there exist $(s+1)$ vectors $w_0,w_1,\ldots,w_s$ in $k^n$ such that not all of them are zero and \eqref{eq_lem_scroll} holds. Then every Kronecker-Weierstrass normal form of $X$ contains a scroll block of length $\le s$. \qed
\end{lem} 

The following result about the lengths of the scroll blocks in Kronecker-Weierstrass normal forms is crucial in the proofs of Theorem \ref{linearly_Koszul} and Theorem \ref{ul_Koszul_scrolls}.
\begin{prop}
\label{KW_form} 
Let $X$ be a Kronecker-Weierstrass matrix and $R$ its determinantal ring. Let $R' = R/(l_1, \ldots, l_r)$ be a quotient ring of $R$ by linear forms $l_1, \ldots, l_r$. Then $R'$ is the determinantal ring of some $2\times e'$ matrix of linear forms $X'$. Moreover, if some Kronecker-Weierstrass normal form of $X'$ has a scroll block of length $s$, then $X$ has a scroll block of length at most $s$.
\end{prop}
\begin{proof} By induction, we may assume that $R' = R /(l)$ for some linear form $l$. We use the notation of Section \ref{background}: the set of variables of $S$ is $\{\xsf_1, \ldots, \xsf_n\}$ with dual basis $\{\xsf^*_1,\ldots,\xsf^*_n\}$.

Assume that $l = \xsf_{i} -\sum_{j > i}a_j \xsf_j$. We call $i$ the leading variable of $l$. We observe that $X'$ is obtained from $X$ by deleting $\xsf_i$ and replacing it by $\sum_{j > i} a_j \xsf_j$. Then $R'$ is clearly the determinantal ring of the matrix $X'$ just described. Let $A, B$ be the matrices corresponding to rows of $X$ as in Section \ref{background}. Also, let $A', B'$ be the matrices corresponding to rows of $X'$. 

{\bf Step 1}: If $X$ is just one block, we show that $X'$ cannot contain any scroll block. 

{\bf Case 1a}: $X$ is one scroll block
\[
\left(\begin{matrix}
           \xsf_1    & \xsf_2     & \ldots          & \xsf_{s-1} & \xsf_s\\
           \xsf_2    & \xsf_3     & \ldots          & \xsf_s     & \xsf_{s+1} \end{matrix}\right).
\]
Now $X'$ is the matrix 
\[
\left(\begin{matrix}
           \xsf_1    & \xsf_2     & \ldots & \xsf_{i-1}                  & \sum_{j=i+1}^{s+1}a_j\xsf_j &\ldots   & \xsf_s\\
       \xsf_2    & \xsf_3     & \ldots & \sum_{j=i+1}^{s+1}a_j\xsf_j   & \xsf_{i+1}                &\ldots   & \xsf_{s+1} \end{matrix}\right).
\]
Hence in the new coordinates $\xsf_1,\ldots,\xsf_{i-1},\xsf_{i+1},\ldots,\xsf_{s+1}$, $A'$ is the following matrix
\[
A'=\left(\begin{matrix} 
E_{i-1} & 0 \\
0       & A''  
\end{matrix}\right)
\]
where $E_{i-1}$ is the unit matrix of size $(i-1)\times (i-1)$ and
$$
A''=
\left(\begin{matrix} 
a_{i+1} & a_{i+2} & \cdots & a_s & a_{s+1} \\
1 & 0 & \cdots & 0 & 0 \\
0 & 1 & \cdots & 0 & 0 \\
\vdots & \vdots & \ddots & \vdots & \vdots \\
0 & 0 & \cdots & 1 & 0  
\end{matrix}\right) \in M_{(s-i+1)\times (s-i+1)}.
$$
Similarly, 
\[
B'=\left(\begin{matrix} 
F & 0 \\
0       & B''  
\end{matrix}\right)
\]
where
\[
F=
\left(\begin{matrix}
0 & 1 & 0 & \cdots  & 0\\
0 & 0 & 1 & \cdots  & 0\\
\vdots &\vdots & \vdots &\ddots  &\vdots\\
0 & 0 & 0      & \cdots & 1
\end{matrix}\right) \in M_{(i-2)\times (i-1)},
\]
and
$$
B''=
\left(\begin{matrix} 
a_{i+1} & a_{i+2} & \cdots & a_s & a_{s+1} \\
1 & 0 & \cdots & 0 & 0 \\
0 & 1 & \cdots & 0 & 0 \\
\vdots & \vdots & \ddots & \vdots & \vdots \\
0 & 0 & \cdots & 1 & 0 \\
0 & 0 & \cdots & 0 & 1 \\
\end{matrix}\right) \in M_{(s-i+2)\times (s-i+1)}.
$$
Therefore the pencil $A' + vB'$ is

\[
A' + vB'=\left(\begin{matrix} 
1 	&v 		&0		&\cdots 	&0     		&\cdots      	&\cdots    	& 0         & 0 \\
0 	&1      		&v 		&\cdots 	&0		&\cdots    	&0         	& 0	&0 \\
\vdots	&\vdots 	&\ddots 	&\ddots  	&\vdots      	& \vdots    	&\vdots  	&0	&0\\
0 	&0		& \cdots 	& 1           	& v  		& 0 		& \cdots 	&0 	&0\\
0 	&0		& \cdots 	& 0		&1           	& v a_{i+1}  	& va_{i+2} 	& \cdots 	& v a_{s+1}\\
0 	&0		& \cdots 	& 0           	&0		& a_{i+1}+v &a_{i+2}   	& \cdots 	& a_{s+1} 	\\
0 	&0		& \cdots 	& 0           	&0		& 1                &v              	&\cdots 	& 0  	\\
\vdots& \vdots 	& \ddots 	& \vdots 	& \vdots    	&\vdots 	& \ddots 	&\ddots	&\vdots\\
0 	&0		& \cdots 	& 0           	&0                 &0		&\cdots     	& 1 		& v 	
\end{matrix}\right) \in M_{s\times s}.
\]
The determinant of $A' + vB'$ is a polynomial of degree $(s - i)$ in $v$ with leading coefficient $1$. Therefore $\rank (A' + vB') = s$. By Lemma \ref{lem_scroll}, any Kronecker-Weierstrass normal form of $X'$ has no scroll blocks.

{\bf Case 1b:} $X$ is one nilpotent block or one Jordan block. In this case, it is easy to see that $A'$ has independent columns. By Lemma \ref{lem_scroll}, every Kronecker-Weierstrass normal form of $X'$ has no scroll blocks. 

{\bf Step 2}: Now assume that $X$ consists of at least $2$ blocks. By induction on the number of blocks we may assume that the leading variable of $l$ is in the set of variables of the first block of $X$. We note that $A, B\in M_{e\times n}$ are block matrices of the following form 
$$A = \begin{pmatrix} A_{11} & 0 \\ 0 & A_{22}\end{pmatrix}, B = \begin{pmatrix} B_{11} & 0 \\ 0 & B_{22}\end{pmatrix}.
$$
Hence $A', B'\in M_{e\times (n-1)}$ are upper block matrices of the form 
$$A' = \begin{pmatrix} A'_{11} & A'_{12}  \\ 0 & A_{22} \end{pmatrix}, B' = \begin{pmatrix} B'_{11} & B'_{12} \\ 0 & B_{22}\end{pmatrix}.$$

Assume that some canonical form of $X'$ has a scroll block. Let $s$ be the shortest length of such a scroll block of $X'$. By Lemma \ref{lem_scroll} (i), there exist $(s+1)$ independent vectors $w'_0, \ldots, w'_{s}\in k^{n-1}$ such that:
$$A'w'_0 = 0, A'w'_1 = B'w'_0,\ldots, A' w'_{s} = B' w'_{s-1}, B' w'_s = 0.$$
For each $i =0, \ldots, s$, write 
$$w_i' = \left( \begin{matrix} u_i' \\ v_i' \end{matrix}\right),$$
where $u_i'$ is a column vector of size equal to the number of columns of $A_{11}'$, and $v_i'$ is a column vector of size equal to the number of columns of $A_{22}$.
 
Let 
$$w_i = \left( \begin{matrix} 0 \\ v_i' \end{matrix}\right) \in k^n,$$ 
where $0$ is the zero vector of size equal to the number of columns of $A_{11}$. From the form of the matrices $A,B, A',B'$ we have 
\begin{equation}\label{scroll_type_eq}
Aw_0 = 0, A w_1 = Bw_0, \ldots, A w_{s} = B w_{s-1}, B w_s = 0.
\end{equation}
If not all of the vectors $w_0,\ldots, w_s$ are zero vectors, by Lemma \ref{lem_scroll} (ii), $X$ has a scroll block of length at most $s$. Assume that all of the vectors $w_0, \ldots, w_s$ are zero vectors. From the equation \eqref{scroll_type_eq} we have 
$$A_{11}' u_0' = 0, A_{11}' u_1' = B_{11}' u_0, \ldots, A_{11}'u_s' = B_{11}' u_{s-1}', B_{11}' u_s' = 0.$$
Moreover, the vectors $u_0', \ldots, u_s'$ are linearly independent. By Lemma \ref{lem_scroll}, the pencil $A_{11}' + vB_{11}'$ has a scroll block. This pencil is obtained by replacing $x_i$ by $\sum_{i<j \le m} a_j x_j$, where $m$ is the last index of the variables appearing in the first block. The last condition contradicts with the case of $X$ consisting of just one block.
\end{proof}
We will also need the information about lengths of nilpotent blocks of linear sections of rational normal scrolls. This will be important for the proofs of Theorem \ref{linearly_Koszul}(i) and (ii) in Section \ref{sect_application}.
\begin{lem}
\label{KW_quotient}
Let $R=R(\nsf_1,\ldots,\nsf_{\tsf})$ be a rational normal scroll where $1\le \nsf_1 \le \cdots \le \nsf_{\tsf}$. Let $Y$ be a subset of the set of natural coordinates of $R$. Then in any Kronecker-Weierstrass normal form of the matrix defining $R/(Y)$, every nilpotent block has length at most $\nsf_{\tsf}$.
\end{lem}
\begin{proof}
Let $X$ be the matrix defining $R/(Y)$. The pencil corresponding to $X$ is a block matrix whose each block is obtained by deleting certain columns corresponding to the variables in $Y$ from the matrix pencil of $R$. Since $k$ is algebraically closed of characteristic $0$, each block in the matrix of $R$ modulo some variables has a Kronecker-Weierstrass normal form. Since the Kronecker-Weierstrass normal form of a block matrix is the concatenation of normal forms of these blocks, it is clear that each nilpotent block in any normal form of $X$ has length at most $\nsf_{\tsf}$. 
\end{proof}

\section{The sufficient condition}
\label{Koszul_sufficient}
In this section, we prove the sufficient condition in Theorem \ref{main}. This is done by
\begin{thm}
\label{Koszul_filtration}
Let $X$ be a concatenation of nilpotent blocks, scroll blocks and Jordan blocks. Assume that $X$ satisfies the {\em length condition} $m \le 2n$, where $m$ is the maximal length of a nilpotent block and $n$ is the minimal length of a scroll block. Then the ring $R = k[X]/I_2(X)$ has a Koszul filtration.
\end{thm}
Although the construction will not be straightforward, the idea behind is quite simple. We start by constructing Koszul filtration for the submatrix of nilpotent and scroll blocks in Section \ref{sect_ns}, and for the submatrix of Jordan blocks in Section \ref{sect_Jordan}. Then ``concatenating" these two filtrations in a suitable way, we get the Koszul filtration for the original matrix. The proof of Theorem \ref{Koszul_filtration} will be given in Section \ref{sect_general}. 

We assume that $X$ has the length sequence
\[
(\underbrace{m_1,\ldots,m_c}_{\Ncc},\underbrace{n_1,\ldots,n_d}_{\Sc},\underbrace{p_{11},\ldots,p_{1g_1},p_{21},\ldots,p_{tg_t}}_{\Jc}).
\]
To simplify the matter, we still use the notation of Section \ref{background} for the blocks and entries of $X$. By abuse of notation, we use $x_{i,j}$, $y_{i,j}$ and $z^i_{j,r}$ to denote the class of $x_{i,j}$, $y_{i,j}$ and $z^i_{j,r}$ in the quotient ring $k[X]/I_2(X)$, respectively. To verify the colon condition in the proof of Theorem \ref{Koszul_filtration}, the following simple identities are useful.

\begin{lem}
\label{lem_identities}
We have the following identities in $R=k[X]/I_2(X)$:
\begin{enumerate}
\item $x_{\pnt,\pnt}z^{\pnt}_{\pnt,\pnt}=0$ and $(x_{1,1},\ldots,x_{c,m_c-1})^2=0$.
\item For all $1\le i\le c, 1\le r\le m_i-1, 1\le j\le d$ and $1\le s\le n_j+1$, if either $r+s\ge m_i+1$ or $r+s\le n_j+1$, then 
$x_{i,r}y_{j,s}=0$.
\item For all $1\le i \le d, 1\le r<s\le n_i+1$, 
\[
(z^{\pnt}_{\pnt,\pnt}) \subseteq  (y_{i,r}):y_{i,s}.
\]
\item For all $1\le i\le d,2\le r\le n_i+1$,
\[
\sum_{j=1}^{d}(y_{j,1},\ldots,y_{j,n_j}) \subseteq  (y_{i,r-1}):y_{i,r}.
\]
\item For all $1\le i\le d, 1\le r \le n_i$,
\[
\sum_{j=1}^{d}(y_{j,2},\ldots,y_{j,n_j+1}) \subseteq  (y_{i,r+1}):y_{i,r}.
\]
\item For all $1 \le i<j \le t$, 
\[ z^i_{\pnt,\pnt} z^j_{\pnt,\pnt} = 0.\]
\end{enumerate}
\end{lem}
\begin{proof}
For (i): for ease of notation, assume that we have a Jordan block and a nilpotent block of $X$ of the form
\[
\left(\begin{matrix}
z_1                   & z_2       & \ldots  & z_{p-1}        & z_p \\
z_2+\lambda z_1   & z_3+\lambda  z_2 & \ldots  & z_p+\lambda  z_{p-1} &\lambda  z_p\end{matrix}\right)
\]
and
\[
\left(\begin{matrix}
       0        & x_1    & x_2     & \ldots          & x_{m-2} & x_{m-1}\\
      x_1   & x_2    & x_3     & \ldots              & x_{m-1}   & 0 \end{matrix}\right),
\]
respectively.

We have that $x_1(z_1,\ldots,z_p)=0$. Then since the $2$-minors
\[
\left(\begin{matrix}
x_1 & z_r \\
x_2 & z_{r+1}+\lambda z_r
\end{matrix}\right) \;\;\; \text{ and }\;\;\;
\left(\begin{matrix}
x_1 & z_p \\
x_2 & \lambda z_p
\end{matrix}\right)
\]
are zero, we get that $x_2(z_1,\ldots,z_p)=0$. Continuing in this manner, we get $x_{\pnt}z_{\pnt}=0$. This gives the first part of (i). The second part is proved similarly. 

For (ii): additionally to the above Jordan and nilpotent blocks, consider a scroll block of $X$ of the form 
\[
\left(\begin{matrix}
      y_1       & y_2      & \ldots          & y_{n-1} & y_n\\
      y_2   & y_3         & \ldots              & y_n   & y_{n+1} \end{matrix}\right),
\]
we want to show that $x_iy_j=0$ if $i+j\le n+1$ or $i+j\ge m+1$. Firstly we have 
\[
x_1y_1=x_1y_2=\cdots=x_1y_n=0.
\]
For $2\le s\le n$, as the minor
\[
\left(\begin{matrix}
x_1 & y_{s-1} \\
x_2 & y_s
\end{matrix}\right)
\]
is zero, we get $x_2y_{s-1}=0$. Continuing in this manner, we get $x_iy_j=0$ if $i+j\le n+1$. Similarly, starting with
\[
x_{m-1}y_2=x_{m-1}y_3=\cdots=x_{m-1}y_{n+1}=0,
\]
we obtain the remaining claim.

For (iii): immediate from looking at the 2-minors of the form
\[
\left(\begin{matrix}
z_i & y_{s-1} \\
z_{i+1}+\lambda z_i & y_s
\end{matrix}\right)
\]
we have $y_s(z_1,\ldots,z_p)\subseteq y_{s-1}(z_1,\ldots,z_p)$. The conclusion follows.

We leave the details of (iv) and (v) to the readers. For (vi), consider another Jordan block of $X$ of the form
\[
\left(\begin{matrix}
u_1                   & u_2       & \ldots  & u_{q-1}        & u_q \\
u_2+\beta u_1   & u_3+\beta  u_2 & \ldots  & u_q+\beta  u_{q-1} &\beta  u_q\end{matrix}\right),
\]
where $\beta \neq \lambda$. We wish to show that $u_iz_j=0$ for all $i,j$.

As the minor
\[
\left(\begin{matrix}
z_p & u_q \\
\lambda z_p & \beta u_q
\end{matrix}\right)
\]
is zero and $\beta-\lambda\neq 0$, we get $z_pu_q=0$. Looking at the minor
\[
\left(\begin{matrix}
z_p & u_{q-1} \\
\lambda z_p & \beta u_{q-1}+u_q
\end{matrix}\right),
\] 
we then obtain $z_pu_{q-1}=0$. Continuing in this manner, we get $z_p(u_1,\ldots,u_q) = 0$. By reverse induction on $1\le j\le p$ we obtain that $z_j(u_1,\ldots,u_q)=0$.
\end{proof}

\subsection{Matrices of nilpotent and scroll blocks}
\label{sect_ns}
First note that the length condition in Theorem \ref{Koszul_filtration} involves only nilpotent and scroll blocks. Hence it is natural to start building the Koszul filtration for the case $X$ contains only such blocks. In this subsection we assume that $X$ is a concatenation of nilpotent blocks and scroll blocks with length sequence
$$ 
(\underbrace{m_1,\ldots,m_c}_{\Ncc},\underbrace{n_1,\ldots,n_d}_{\Sc}).
$$
Moreover, assume that $m_c \le 2 n_1$. 

The following special case is enough to illustrate the construction of the Koszul filtration.
\begin{ex}
Let $X$ be the matrix of one nilpotent and one scroll block satisfying the length condition. Hence
\[
X=\left(\begin{matrix}
       0   & x_1 & x_2  &\ldots  & x_{m-1}\\ 
       x_1 & x_2 & x_3  &\ldots  & 0 \end{matrix}\right. \left. \dashline ~ \begin{matrix}
                                       y_1 & y_2 &\ldots &y_n \\ 
                                       y_2 & y_3 &\ldots &y_{n+1}
                                       \end{matrix} \right).
\]
where $2\le m\le 2n$. We have
\begin{align*}
I_2(X)=&(x_1,\ldots,x_{m-1})^2+(x_iy_j:i+j\le n+1~ \textnormal{or $i+j\ge m+1$}) \nonumber \\
&+(x_iy_j-x_{i+1}y_{j-1}: n+2\le i+j \le m) \nonumber \\
&+ (y_iy_j-y_{i+1}y_{j-1}: 1\le i \le j \le n+1).
\end{align*} 
Then $R=k[X]/I_2(X)$ has a Koszul filtration as follow. Let $s=\max\{m-n,1\}$. Define $\Fc=\{H_0,\ldots,H_{m-1}\}~ \bigcup ~\{I_{a,b}:b\ge 0, 1\le a \le n+1-b\}$ where
\begin{align*}
H_0=(0), H_1=(x_s),H_2=(x_{s-1},x_s),\ldots,H_s=(x_1,\ldots,x_s),\\
H_{s+1}=(x_1,\ldots,x_s,x_{s+1}),\ldots,H_{m-1}=(x_1,\ldots,x_s,\ldots,x_{m-1}),\\
I_{a,b}=H_{m-1}+(y_1,y_2,\ldots,y_a,y_{n+2-b},y_{n+3-b},\ldots,y_n,y_{n+1}).
\end{align*}
Then $\Fc$ is  a Koszul filtration for $R$.

To be more precise, note that $I_{n+1,0}=\mm$. For the required colon condition, we can check the following identities:
\begin{enumerate}
\item $H_0:H_1=I_{n+1-s,1} \label{H0_eq1}$.\\
\item $H_1:H_2=\cdots=H_{s-1}:H_s=H_s:H_{s+1}=\cdots=H_{m-2}:H_{m-1}=\mm \label{Hs_eq1}$.\\
\item $\textnormal{If $b\ge 2$, then}~ I_{a,b-1}:I_{a,b}=\mm$.\\
\item $\textnormal{If $a\ge 2$, then} ~I_{a-1,b}:I_{a,b}=\begin{cases}
\mm, &\textnormal{if $b\ge 1$};\\
I_{n,0}, &\textnormal{if $b=0$}.
\end{cases}$\\
\item $\textnormal{If $a=1,b=0$, then}~ H_{m-1}:I_{1,0}=H_{m-1}$.\\
\item $\textnormal{If $a=1,b=1$, then}~ I_{1,0}:I_{1,1}=I_{1,0}$.
\end{enumerate}
These identities will be justified by the forthcoming lemmas of this section.
\end{ex}
Let us comeback to the general case of matrices with only nilpotent and scroll blocks. To facilitate the presentation, it is useful to introduce the following notion.
\begin{defn}
We say that a sequence $\bsb=(b_1,b_2,\ldots,b_s)$ of non-negative integers {\em has no gap} if for any $1\le i\le s$, $b_i=0$ implies that $b_{i+1}=\cdots=b_s=0.$
\end{defn}

Our Koszul filtration for $R = k[X]/I_2(X)$ consists of the ideals of the following types. 

\begin{constr}[Koszul filtration for matrices of nilpotent and scroll blocks]
\label{constr_ns_filtration} 
For each $i = 1,\ldots,c$, denote $s_i = \max\{m_i - n_1,1\}$. Consider ideals of the following types:

\begin{enumerate}
\item $H_{0,m_0-1} = (0)$ (where $m_0$ is used just for systematic reason),
\item $H_{i,r}$, where $1\le i \le c, 1\le r\le m_i-1$, given recursively by
\begin{gather*}
H_{i,1}=H_{i-1,m_{i-1}-1}+(x_{i,s_i}),\\
H_{i,2}=H_{i-1,m_{i-1}-1}+(x_{i,s_i-1},x_{i,s_i}),\ldots,H_{i,s_i}= H_{i-1,m_{i-1}-1}+(x_{i,1},\ldots,x_{i,s_i}),\\
H_{i,s_i+1}=H_{i-1,m_{i-1}-1}+(x_{i,1},\ldots,x_{i,s_i},x_{i,s_i+1}),\ldots,\\
H_{i,m_i-1}=H_{i-1,m_{i-1}-1}+(x_{i,1},x_{i,2},\ldots,x_{i,m_i-1}).
\end{gather*}

\item and $I_{s;\bsa,\bsb}$, where $1\le s\le d$, and $\bsa = (a_1,\ldots,a_s), \bsb = (b_1,\ldots,b_s)$ such that 
$1\le a_j\le n_j+1-b_j$ for $1\le j\le s$, $\bsb$ has no gap, given by
\[
I_{s;\bsa,\bsb}=H_{c,m_c-1} +\sum_{j=1}^{s} \left[(y_{j,1},y_{j,2},\ldots,y_{j,a_j}) + (y_{j,n_j-b_j+2},y_{j,n_j-b_j+3},\ldots,y_{j,n_j+1})\right].
\]
\end{enumerate}
Of course, if there is no nilpotent block then there is only one ideal of type $H$ which is $H_{0,m_0-1}=0$, and similar convention works if there is no scroll block.
\end{constr}
\begin{rem}
If $X$ consists only of scroll blocks, namely $X$ defines a rational normal scroll, then we obtain from the construction a Koszul filtration for that scroll. This gives new information about the Koszul property of rational normal scrolls.
\end{rem}

The fact that the ideals $H_{i,r}$, and $I_{s;\bsa,\bsb}$ form a Koszul filtration for $R$ follows from the following series of lemmas. 

Firstly, for $1\le i\le c,1\le j\le d$, define $a_{i,j},b_{i,j}$ as follow: $a_{i,j} = n_j + 1 - s_i$ and $b_{i,j}=\min\{n_j +1 +s_i - m_i,s_i\}$. Concretely,
\begin{enumerate}
\item if $m_i\ge n_j+2$ then $b_{i,j} = n_j +1 +s_i - m_i$,
\item if $m_i\le n_j+1$ then  $b_{i,j}=s_i$.
\end{enumerate}
In any case, we have $b_{i,j}\ge 1$ and $1\le a_{i,j}\le n_j+1-b_{i,j}$. Indeed, since $m_c\le 2n_1$, we get $s_i=\max\{m_i-n_1,1\} \le n_1$, so $a_{i,j}\ge 1$. Also, $n_j+1+(m_i-n_1) -m_i \ge 1$, hence $b_{i,j}\ge 1$.
\begin{lem}[Colon condition for the ideals $H_{i,j}$]
\label{H_series}
The following equalities hold for each $1\le i\le c$: 
\begin{enumerate}
\item $H_{i-1,m_{i-1}-1}:x_{i,s_i}=I_{d;\bsa_i,\bsb_i}$, where $\bsa_i=(a_{i,1},\ldots,a_{i,d}),\bsb_i=(b_{i,1},\ldots,b_{i,d})$.
\item $H_{i,j}:H_{i,j+1}=\mm$, for $j = 1, \ldots, m_i-2$.
\end{enumerate}
\end{lem}
\begin{proof}
{\bf For (i)}: firstly, the left-hand side contains the right-hand side. Indeed, take $1\le j\le d$ and $1\le s\le n_j+1$. If $s\le n_j+1-s_i$ then $s+s_i\le n_j+1$, so $y_{j,s}x_{i,s_i}=0$ by Lemma \ref{lem_identities}(ii). Now we show that if $0\le s\le b_{i,j}$ then $y_{j,n_j+2-s}\in H_{i-1,m_{i-1}-1}:x_{i,s_i}$.

If $m_i\ge n_j+2$ and $s\le b_{i,j}=n_j+1+s_i-m_i$ then $n_j+2-s+s_i\ge m_i+1$, and $n_j+2-s\ge m_i-s_i+1\ge 2$, so $y_{j,n_j+2-s}x_{i,s_i}=0$ by Lemma \ref{lem_identities}(ii). On the other hand, if $m_i\le n_j+1$ and $s\le b_{i,j}=s_i$ then $n_j+2-s+s_i\ge n_j+2\ge m_i+1$ so again $y_{j,n_j+2-s}x_{i,s_i}=0$ by Lemma \ref{lem_identities}(ii).

For the reverse inclusion, working modulo $H_{i-1,m_{i-1}-1}$, we can assume that $i=1$. Denote $\bsa=\bsa_1,\bsb=\bsb_1$, what we need to show is
\begin{equation}
\label{eq_colon_x}
0:x_{1,s_1}=I_{d;\bsa,\bsb}.
\end{equation}
To establish \eqref{eq_colon_x}, we will show the equality of the Hilbert series of the two sides.

Consider the short exact sequence
\[
0\to R/(0:x_{1,s_1})(-1)\xrightarrow{\cdot x_{1,s_1}} R  \to R/(x_{1,s_1})\to 0.
\]
Denote $m=\max\{m_1,\ldots,m_c\}$. Let $R'$ be the determinantal ring of the submatrix of $X$ consisting of scroll blocks. By Theorem \ref{Hilbert_series}, we have
\[
H_{R}(v)=(m_1+\cdots+m_c-c)v +\sum_{q=2}^m\sum_{i=1}^{c}\sum_{r=0}^{m_i-2}N(n_1,\ldots,n_d,m_i-1-r;q)v^q+H_{R'}(v).\]

The length sequence of $R/(x_{1,s_1})$ is
\[
\underbrace{s_1,m_1-s_1,m_2,\ldots,m_c}_{\Ncc},\underbrace{n_1,\ldots,n_d}_{\Sc}.
\]
A small remark here is that $s_1=\max\{m_1-n_1,1\}\le m_1-s_1$. Now from $m_1\le 2n_1$, it is clear that $s_1,m_1-s_1\le n_1$. Therefore by Lemma \ref{lem_N_function} and Theorem \ref{Hilbert_series}, we get
\begin{gather*}
H_{R/(x_{1,s_1})}(v)=\\
(m_1+\cdots+m_c-c-1)v +\sum_{q=2}^m\sum_{i=2}^{c}\sum_{r=0}^{m_i-2}N(n_1,\ldots,n_d,m_i-1-r;q)v^q +H_{R'}(v).
\end{gather*}
Together with the formula for $H_{R}(v)$, we infer
\[
H_{R}(v)-H_{R/(x_{1,s_1})}(v)=v+\sum_{q=2}^m\sum_{r=0}^{m_1-2}N(n_1,\ldots,n_d,m_1-1-r;q)v^q.
\]
Note that if $q\ge 3$ then $N(n_1,\ldots,n_d,m_1-1-r;q)=0$ for all $r\ge 0$. Indeed, we have $m_1-1-r\le 2n_1\le (q-1)n_1$, so the conclusion holds because of Lemma \ref{lem_N_function}.

{\bf Claim}: If $q=2$ then 
\[
\sum_{r=0}^{m_1-2}N(n_1,\ldots,n_d,m_1-1-r;2)v^2=\left(\sum_{j:~m_1\ge n_j+2}(m_1-n_j-1)\right)v^2.
\]

{\em Proof}: If a sequence $(\vsf_1,\ldots,\vsf_d)$ of non-negative integers satisfies $\sum_{j=1}^dn_j\vsf_j\le m_1-2-r$ and 
$\sum_{j=1}^{d}\vsf_j=2-1=1$ then exactly one of $\vsf_1,\ldots,\vsf_d$ equals to $1$ and the others are zero. Fix $1\le j\le d$, then the equality $\vsf_j=1$ happens if and only if $n_j\le m_1-2-r$, namely if and only if $m_1\ge n_j+2$, and there are exactly $(m_1-n_j-1)$ values of $r$ such that this is the case. Therefore the claim is proved.

From these facts, we obtain
\[
H_{R}(v)-H_{R/(x_{1,s_1})}(v)=v+\left(\sum_{j:~m_1\ge n_j+2}(m_1-n_j-1)\right)v^2.
\]
Hilbert series is additive along short exact sequences, so
\begin{equation}
\label{eq_H_diff}
H_{R/(0:x_{1,s_1})}=1+\left(\sum_{j:~m_1\ge n_j+2}(m_1-n_j-1)\right)v.
\end{equation}
Note that 
\[
(y_{j,1},\ldots,y_{j,n_j+1-s_1},y_{j,n_j+2-b_{1,j}},y_{j,n_j+3-b_{1,j}},\ldots,y_{j,n_j+1})=(y_{j,1},y_{j,2},\ldots,y_{j,n_j+1})
\] 
unless $n_j+1-s_1\le n_j-b_{1,j}$, namely $b_{1,j}\le s_1-1$, which is nothing but $m_1\ge n_j+2$. Therefore $R/I_{d;\bsa,\bsb}$ has the length sequence
\[
\underbrace{m_1-n_j: ~\textnormal{where $m_1\ge n_j+2$}}_{\Ncc}.
\]
Applying Theorem \ref{Hilbert_series}, we infer
\[
H_{R/I_{d;\bsa,\bsb}}(v)=1+\left(\sum_{j:~m_1\ge n_j+2}(m_1-n_j-1)\right)v.
\]
Therefore combining with \eqref{eq_H_diff}, $H_{R/I_{d;\bsa,\bsb}}(v)=H_{R/(0:x_{1,s_1})}(v)$ and \eqref{eq_colon_x} is true.

{\bf For (ii)}: Modulo $H_{i,j}$ one reduces to the case when the first nilpotent block of $X$ has length $m_1 \le n_1$. What we have to prove is
$$0:x_{1,1}=\mm.$$
This follows from part (i), as in this case $\bsa = (n_1, \ldots, n_d)$ and $\bsb = (1, \ldots, 1)$. 
\end{proof}

In the following two lemmas, working modulo $H_{c,m_c-1}$, we assume that $X$ has no nilpotent blocks. For simplicity, for each $s$, $1 \le s \le d$, we denote 
$$1_s = (\underbrace{1, \ldots, 1}_{s~ \textnormal{times}}), 0_s = (\underbrace{0, \ldots, 0}_{s~ \textnormal{times}}).$$

\begin{lem}[Colon condition for the ideals $I_{s;\bsa,\bsb}$ where $\max_{1\le i\le s}\{a_i,b_i\}\ge 2$]
\label{I>=2_series} 
Assume that $1\le s\le d$ and $\bsa=(a_1,\ldots,a_s),\bsb=(b_1,\ldots,b_s)$ be such that $a_1,\ldots,a_s\ge 1$ and $\bsb$ has no gap.
\begin{enumerate}
\item
If $b_i\ge 2$ for some $1\le i\le s$, denote $\widehat{\bsb}=(b_1,\ldots,b_{i-1},b_i-1,b_{i+1},\ldots,b_s)$. Then
$I_{s;\bsa,\bsb}=I_{s,\bsa,\widehat{\bsb}}+(y_{i,n_i-b_i+2})$ and
\[
I_{s,\bsa,\widehat{\bsb}}:y_{i,n_i-b_i+2}=\mm.
\]
\item If $b_1,\ldots,b_s\le 1$ and $a_i\ge 2$ for some $1\le i\le s$, denote $\widehat{\bsa}=(a_1,\ldots,a_{i-1},a_i-1,a_{i+1},\ldots,a_s)$. Then $I_{s;\bsa,\bsb}=I_{s,\widehat{\bsa},\bsb}+(y_{i,a_i})$ and
\[
I_{s;\widehat{\bsa},\bsb}:y_{i,a_i}=\begin{cases}
\mm, &\textnormal{if $b_i= 1$};\\
I_{d;(a'_1,\ldots,a'_{i-1},a'_i,\ldots,a'_d),0_d}, &\textnormal{if $b_i=0$},
\end{cases}
\]
where $a'_i=n_i$ and for $j\neq i$, $a'_j=\begin{cases}
n_j,  &\textnormal{if $n_j-a_j\ge n_i-a_i+1$},\\
n_j+1, &\textnormal{otherwise}.
\end{cases}$
\end{enumerate}
\end{lem}
\begin{proof} 
{\bf For (i)}: By Lemma \ref{lem_identities}(iii)-(v) we have 
$$
\sum_{r=1}^{d}(y_{r,2},\ldots,y_{r,n_r+1}) \subseteq (y_{i,1},y_{i,n_i-b_i+3}):y_{i,n_i-b_i+2}.
$$
Therefore it is enough to show that $y_{r,1}\in I_{s,\bsa,\widehat{\bsb}}:y_{i,n_i-b_i+2}$ for all $1\le r \le d$. As $a_i\ge 1$ for $1\le i \le s$, we only need to prove that $y_{r,1}\in I_{s,\bsa,\widehat{\bsb}}:y_{i,n_i-b_i+2}$ for $s+1\le r \le d$. This is true since $n_i-b_i+2\le n_i\le n_r$ and hence
\[
y_{r,1}y_{i,n_i-b_i+2}=y_{r,n_i-b_i+2}y_{i,1} \in (y_{i,1}).
\]

{\bf For (ii)}: firstly assume that $b_i= 1$, hence $y_{i,n_i+1}\in I_{s;\widehat{\bsa},\bsb}$. By Lemma \ref{lem_identities}(iii)-(iv), we only need to check that $y_{j,n_j+1}\in I_{s;\widehat{\bsa},\bsb}:y_{i,a_i}$ for all $1\le j\le d$. For each $j\le i$, since $\bsb$ has no gap, $b_j= 1$, so $y_{j,n_j+1}\in I_{s;\widehat{\bsa},\bsb}$. For $j\ge i+1$, $y_{i,a_i}y_{j,n_j+1}=y_{i,n_i+1}y_{j,n_j+a_i-n_i}$, hence $y_{j,n_j+1} \in (y_{i,n_i+1}):y_{i,a_i}$. This gives us the desired equality.

Secondly, assume that $b_i=0$. We wish to prove
\begin{equation}
\label{eq_I}
I_{s;\widehat{\bsa},\bsb}:y_{i,a_i} = I_{d;(a'_1,\ldots,a'_{i-1},n_i,a'_{i+1},\ldots,a'_d),0_d}.
\end{equation}
If $n_j-a_j\le n_i-a_i$ for some $j\neq i$, then $y_{j,n_j+1}\in I_{s;\widehat{\bsa},\bsb}:y_{i,a_i}$ because $y_{j,n_j+1}y_{i,a_i}=y_{j,a_j}y_{a_i+n_j-a_j+1}\in (y_{j,a_j})$. Combining with Lemma \ref{lem_identities}, we see that the left-hand side contains the right-hand side. Working modulo the ideal 
\[
\sum_{\ell\neq i:~n_{\ell}-a_{\ell}\le n_i-a_i}(y_{\ell,1},\ldots,y_{\ell,n_{\ell}+1}),
\] 
we can assume that $n_j-a_j\ge n_i-a_i+1$ for all $j\neq i$, $1\le j\le s$. The equation \eqref{eq_I} that we have to prove becomes
\[
I_{s;\widehat{\bsa},0_s}:y_{i,a_i} = I_{d;n_1,\ldots,\ldots,n_d,0_d}.
\]
To prove this we use the monoid presentation of a rational normal scroll. Thus we can identify $y_{j,r}$ with $x^{n_j-r+1}y^{r-1}s_j \in k[x,y,s_1,\ldots,s_d]$ for all $1\le j\le d,1\le r\le n_j+1$. Here $x,y,s_1,\ldots,s_d$ are distinct variables. Assume that there exists a polynomial $f$ in the variables $y_{1,n_1+1},\ldots,y_{d,n_d+1}$ such that $fy_{i,a_i} \in I_{s;\widehat{\bsa},0_s}$. Using the monoid grading, we can assume that $f$ is a monomial $\prod_{j=1}^{d}y_{j,n_j+1}^{m_j}$ where $m_j\ge 0$. In the monoid presentation, we have $\prod_{j=1}^{d}(y^{n_j}s_j)^{m_j}x^{n_i+1-a_i}y^{a_i-1}s_i$ belongs to the ideal
\[
\sum_{r\neq i}(x^{n_r}s_r,\ldots,x^{n_r-a_r+1}y^{a_r-1}s_r)+(x^{n_i}s_i,\ldots,x^{n_i-a_i+2}y^{a_i-2}s_i).
\]
This is a contradiction as in the monoid ring $k[x^{n_1}s_1,x^{n_1-1}ys_1,\ldots,y^{n_d}s_d]$, the element $\prod_{j=1}^{d}(y^{n_j}s_j)^{m_j}x^{n_i+1-a_i}y^{a_i-1}s_i$ is not divisible by any monomial generator of the above ideal (by looking at the power of $x$). We conclude the proof of the lemma.
\end{proof}

\begin{lem}[Colon condition for $I_{s;\bsa,\bsb}$ where $\max_{1\le i\le s}\{a_i,b_i\}\le 1$]
\label{I<=1_series}
Assume that $1\le s\le d$, $\bsa=(a_1,\ldots,a_s)$ and $\bsb=(b_1,\ldots,b_s)$ be such that $a_1=\cdots=a_s=1$ and $b_j\le 1$ for all $1\le j\le s$. Denote by $i$ the largest index such that $b_i= 1$. 
\begin{enumerate}
\item If $i\ge 1$, let $\widetilde{\bsb}=(b_1,\ldots,b_{i-1},0,\ldots,0)$. Then $I_{s;1_s,\bsb}=I_{s;1_s,\widetilde{\bsb}}+(y_{i,n_i+1})$ and
\begin{equation}
\label{eq_I1}
I_{s;1_s,\widetilde{\bsb}}:y_{i,n_i+1} = I_{d;(n_1+1,\ldots,n_{i-1}+1,1,n_{i+1}-n_i+1,\ldots,n_d-n_i+1),0_d}.
\end{equation}
\item If $i=0$, then $I_{s;1_s,0_s}=I_{s-1;1_{s-1},0_{s-1}}+(y_{s,1})$ and
\begin{equation}
\label{eq_I2}
I_{s-1;1_{s-1},0_{s-1}}:y_{s,1}=I_{s-1;(n_1+1,\ldots,n_{s-1}+1),0_{s-1}}.
\end{equation}
\end{enumerate}
\end{lem}
\begin{proof}
{\bf For (i)}: First we prove that the left-hand side of \eqref{eq_I1} contains the right-hand side. For each $1 \le j \le i-1$ and each $1\le r \le n_j+1$, we have 
$$y_{i,n_i+1} y_{j,r} = y_{i,n_i} y_{j,r+1} = \cdots = y_{i,n_i-n_j +r} y_{j,n_j+1} \in I_{s;1_s,\widetilde{\bsb}},$$
hence $y_{j,r}\in I_{s;1_s,\widetilde{\bsb}}:y_{i,n_i+1}$.

For each $i \le j \le d$, and each $1 \le r \le n_j +1 - n_i$, we have 
$$
y_{i,n_i+1} y_{j,r} = y_{i,n_i} y_{j,r+1} = \cdots = y_{i,1} y_{j,n_i+r} \in I_{s;1_s,\widetilde{\bsb}},
$$
so $y_{j,r} \in I_{s;1_s,\widetilde{\bsb}}:y_{i,n_i+1}$. Combining with Lemma \ref{lem_identities}(iii), we see that the left-hand side contains the right-hand side. Working modulo the ideal
\[
\sum_{j=1}^{i-1}(y_{j,1},\ldots,y_{j,n_j+1}),
\]
we may assume that $i = 1$. The equation \eqref{eq_I1} that we have to prove becomes
\begin{equation}
\label{eq_colon_y}
I_{s;1_s,0_s}:y_{1,n_1+1}=I_{d;\bsc,0_d}.
\end{equation}
where $\bsc = (1, n_2 - n_1 + 1, \ldots, n_d-n_1+1)$. Modulo $I_{d;\bsc,0_d}$, and after re-indexing the variables, $X$ is a concatenation of Jordan blocks with eigenvalue $0$ and length sequence $\underbrace{(n_1,\ldots, n_1)}_{\Jc}$, and we need to prove that $z^1_{1,1}$ is a non-zero divisor on $k[X]/I_2(X)$. This follows from Lemma \ref{lem_Groebner}, as in this case the $2\times 2$ minors of $X$ form a quadratic Gr\"obner basis for $I_2(X)$ with respect to the graded reverse lexicographic order. In particular, $z^1_{1,1}$ is a non-zero divisor.

{\bf For (ii)}: First we prove that the left-hand side of \eqref{eq_I2} contains the right-hand side. For each $1 \le \ell \le s-1$ and each $2\le j \le n_{\ell} + 1$, we have 
$$
y_{s,1} y_{\ell,j} = y_{s,2} y_{\ell,j-1} = \cdots = y_{s,j} y_{\ell,1} \in I_{s-1;1_{s-1},0_{s-1}}.
$$
Note that the assumption that $n_1 \le n_2 \le \cdots \le n_d$ is essential here, since we need $y_{s, j}$ to be in our set of variables. 

Working modulo the right-hand side, it remains to prove the statement for $X$ being a rational normal scroll and $s = 1$, i.e., $y_{1,1}$ is a non-zero divisor. This is obvious since the corresponding determinantal ring is a domain.
\end{proof}

\subsection{Matrices of Jordan blocks}
\label{sect_Jordan}
The second step is to find Koszul filtrations for concatenations of Jordan blocks. Assume that $X$ is a concatenation of $g_i$ Jordan blocks with eigenvalue $\lambda_i$, for $i = 1, \ldots, t$. Here, we assume that $\lambda_1, \lambda_2, \ldots, \lambda_t$ are the pairwise distinct eigenvalues of blocks of our matrix. The Jordan blocks with the same eigenvalues $\lambda_i$ are arranged in the order of decreasing length. Concretely, 
\[
X=\left(\begin{matrix} X^{1}_1 & X^1_2 & \cdots & X^1_{g_1} \dashline \cdots \dashline X^t_1 & X^t_2 & \cdots X^t_{g_t}                      \end{matrix}\right),
\]                  
where 
\[
X^i_j=\left(\begin{matrix} z^i_{j,1}     & z^i_{j,2}  &\ldots & z^i_{j,p_{ij}} \\ 
                       z^i_{j,2} + \lambda_i z^i_{j,1}     & z^i_{j,3} + \lambda_i z^i_{j,2}  &\ldots &\lambda_i z^i_{j,p_{ij}}
                                       \end{matrix}\right).
\]

\begin{constr}[Koszul filtration for matrices of Jordan blocks]
\label{constr_Jordan_filtration} 
Our Koszul filtration will consist of the ideals of the following types:

\begin{enumerate}
\item $J^{0, g_0, p_{0g_0}} = (0)$ (where $g_0,p_{0g_0}$ are used just for systematic reason),
\item $J^{i,j,r}$, where $1\le i \le t$, $1 \le j \le g_i$ and $1\le r \le p_{ij}$, 
$$J^{i,j,r}=(z^i_{1,1},z^i_{1,2},\ldots,z^i_{1,p_{i1}},\ldots,z^i_{j,1},\ldots,z^i_{j,r}).$$

\item and $K^{\ell,i,j,r}$, where $1\le \ell \le t$, $1 \le i\le \ell$, $1\le j \le g_i$ and $1\le r \le p_{ij}$, 
$$K^{\ell,i,j,r} = \mathop{\sum_{1\le u \le \ell}}_{u\neq i} J^{u,g_u,p_{ug_u}} + J^{i,j,r}.$$
\end{enumerate}
 By convention, $J^{i,j,r}=0$ if $i=0$, and 
 $$
 K^{\ell,i,j,r}= \mathop{\sum_{1\le u \le \ell}}_{u\neq i} J^{u,g_u,p_{ug_u}}
 $$ 
 if $j=0$.
\end{constr}
\begin{ex}
Let $X$ be the following concatenation matrix (where $p,q\ge 1, \lambda\in k\setminus 0$)
\[
X=\left(\begin{matrix}
       z_1 & z_2  &\ldots  & z_{p-1} & z_p\\ 
       z_2 & z_3  &\ldots  & z_p     & 0 \end{matrix}\right. \left. \dashline ~ \begin{matrix}
                                       u_1 & u_2 &\ldots &u_{q-1} & u_q \\ 
                                       u_2+\lambda u_1 & u_3+\lambda u_2 &\ldots & u_q+\lambda u_{q-1} &\lambda u_q
                                       \end{matrix} \right).
\]
Then $I_2(X)=I_2(z)+I_2(u)+(z_1,\ldots,z_p)(u_1,\ldots,u_q)$. Here $I_2(z)$ is the ideal of $2$-minors of the first Jordan block of $X$ and similarly for $I_2(u)$, which is also the ideal of $2$-minors of
\[
\left(\begin{matrix}
       u_1 & u_2  &\ldots  & u_{q-1} & u_q\\ 
       u_2 & u_3  &\ldots  & u_q     & 0 \end{matrix} \right).
\]
In this case $t=2$, $g_1=g_2=1$. Consider the following ideals of $k[X]/I_2(X)$:
\begin{align*}
& J^{0,0}=(0),\\
& J^{1,r}=(z_1,\ldots,z_r), J^{2,s}=(u_1,\ldots,u_s),\\
& K^{2,1,r}=(u_1,\ldots,u_q)+(z_1,\ldots,z_r),\\
& K^{2,2,s}=(z_1,\ldots,z_p)+(u_1,\ldots,u_s),
\end{align*}
where $1\le r\le p, 1\le s\le q$. Then the collection 
$$\{J^{0,0}\} ~ \bigcup ~ \{J^{1,r}\} ~ \bigcup ~ \{J^{2,s}\} ~ \bigcup ~ \{K^{2,1,r}\} ~ \bigcup ~ \{K^{2,2,s}\}$$ 
is a Koszul filtration for the ring in question.

In more details, we have $K^{2,1,p}=\mm$. The colon condition is verified by the following equalities
\begin{enumerate}
\item $J^{0,0}:J^{1,1}=J^{2,q}$,
\item $J^{0,0}:J^{2,1}=J^{1,p}$,
\item $J^{1,r-1}:J^{1,r}=J^{2,s-1}:J^{2,s}=\mm$, if $r,s\ge 1$,
\item $J^{2,q}:K^{2,1,1}=J^{2,q}$,
\item $J^{1,p}:K^{2,2,1}=J^{1,p}$,
\item $K^{2,1,r-1}:K^{2,1,r}=K^{2,2,s-1}:K^{2,2,s}=\mm$, if $r,s\ge 1$.
\end{enumerate}
These identities will be justified by the next result.
\end{ex}
The fact that the ideals $\{J^{i,j,r}\} ~ \bigcup ~ \{K^{\ell,i,j,r} \}$ in the Construction \ref{constr_Jordan_filtration} form a Koszul filtration for the determinantal ring $k[X]/I_2(X)$ follows from the following lemma. Note that (i), (ii), (iii) gives the colon condition for $J^{i,j,r}$ with either $r=j=1$ or $r=1, j>1$, or $r>1$, respectively, hence we obtain the colon condition for all ideals of type $J$. Similarly, thanks to (iv), (v), (vi), we obtain the colon condition for all ideals of type $K$.

\begin{lem}[Colon condition for the ideals $J^{i,j,r}$ and $K^{\ell,i,j,r}$]
\label{Jordan_filtration} For each $1 \le \ell \le t$, $1 \le i \le \ell$, $2 \le j \le g_i$, and $2 \le r \le p_{ij}$, there are equalities:
\begin{enumerate} 
\item $J^{i-1,g_{i-1},p_{(i-1)g_{i-1}}}:z^i_{1,1}=K^{t,i,0,0}$,
\item $J^{i,j-1,p_{i(j-1)}} :z^i_{j,1} = K^{t,i,j-1,p_{i(j-1)}}$,
\item $J^{i,j,r-1}:z^i_{j,r} = \mm$,
\item $K^{\ell-1,i-1,g_{i-1},p_{(i-1)g_{i-1}}}:z^i_{1,1}=K^{t,i,0,0},$
\item $K^{\ell,i,j-1,p_{i(j-1)}}:z^i_{j,1} = K^{t,i,j-1,p_{i(j-1)}}$,
\item $K^{\ell,i,j,r-1} :z^i_{j,r} =\mm$.
\end{enumerate}
\end{lem}
\begin{proof} {\bf For (i)}: By Lemma \ref{lem_identities}(v), the left-hand side contains the right-hand side. Working modulo the right-hand side, we may assume that $X$ consists of Jordan blocks with the same eigenvalue $\lambda$ (which can be taken to be $0$) and $i = 1$. What we need to prove is that:
\begin{equation}
0:z^1_{1,1} = 0.
\end{equation}
This follows from the same Gr\"obner basis argument as in the proof of Lemma \ref{I<=1_series}(i).

By the same arguments, we obtain (ii), (iv) and (v).

{\bf For (iii)}: this is a consequence of Lemma \ref{lem_identities}(iii) and the following analogue of \ref{lem_identities}(iv):
\[
\sum_{t=1}^{g}(z_{t,1},\ldots,z_{t,p_t})\subseteq (z_{i,r-1}):z_{i,r}.
\]
By the same arguments, we obtain (vi). The lemma follows.
\end{proof}

\subsection{Koszul filtration for Theorem \ref{Koszul_filtration}}
\label{sect_general}
Let $X$ be a Kronecker-Weierstrass matrix satisfying the length condition. The final step to get a Koszul filtration for the determinantal ring of $X$ is ``concatenating" the filtration for Jordan blocks in Section \ref{sect_Jordan} with the above Koszul filtration for the matrix of nilpotent and scroll blocks in Section \ref{sect_ns}. The result is our desired filtration for any Kronecker-Weierstrass matrix satisfying the length condition. 

\begin{constr}[Koszul filtration]
\label{constr_filtration} 
For each $i = 1,\ldots,c$, denote $s_i = \max\{m_i - n_1,1\}$. With notation from Sections \ref{sect_ns} and \ref{sect_Jordan}, our Koszul filtration consists of the ideals of the following types:

\begin{enumerate}
\item $H_{0,m_0-1}=(0)$,

\item $H_{i,r}$ (where $1\le i \le c, 1\le r\le m_i-1$) with generators as in Construction \ref{constr_ns_filtration},

\item $I_{s;\bsa,\bsb}$ (where $1\le s\le d$, and $\bsa = (a_1,\ldots,a_s), \bsb = (b_1,\ldots,b_s)$ such that 
$b_j\ge 0, 1\le a_j\le n_j+1-b_j$ for $1\le j\le s$, $\bsb$ has no gap) with generators as in Construction \ref{constr_ns_filtration},

\item $J_{\bsa,\bsb}^{i,j,r}$, where $1 \le i \le t$, $1 \le j \le g_i$, $1 \le r \le p_{ij}$, and $\bsa = (a_1, \ldots,a_d)$, $\bsb = (b_1, \ldots, b_d)$ such that  $b_r\ge 0, 1\le a_r\le n_r+1-b_r$ for $r=1,\ldots,d$, $\bsb$ has no gap, given by
\[
J_{\bsa,\bsb}^{i,j,r} = I_{d;\bsa,\bsb} + J^{i,j,r}.\]
Here $J^{i,j,r}$ has generators as in Construction \ref{constr_Jordan_filtration}.
\item and $K_{\bsa,\bsb}^{\ell, i,j,r}$, where $1 \le i\le  \ell \le t$, $1 \le j\le g_i$, $1 \le r \le p_{ij}$ and $\bsa = (a_1, \ldots,a_d)$, $\bsb = (b_1, \ldots, b_d)$ such that $b_r\ge 0, 1\le a_r\le n_r+1-b_r$ for $r=1,\ldots,d$, $\bsb$ has no gap, given by
\[
K_{\bsa,\bsb}^{\ell,i,j,r} = I_{d;\bsa,\bsb} + K^{\ell,i,j,r}.
\]
Here $K^{\ell,i,j,r}$ has generators as in Construction \ref{constr_Jordan_filtration}.
\end{enumerate}
\end{constr}
\begin{rem}
1) Construction \ref{constr_filtration} generalizes Construction \ref{constr_ns_filtration} and Construction \ref{constr_Jordan_filtration}.

2) Note that if $I$ is an ideal in a Koszul filtration of $R$ then necessarily $R/I$ is Koszul. This can be used as a test for our Koszul filtration: one can check that modulo ideals of type $H, I, J$ or $K$, we again get determinantal rings of matrices {\em satisfying the length condition}. Therefore such quotient rings should also be Koszul by Theorem \ref{Koszul_filtration}, giving support to the correctness of our filtration \ref{constr_filtration}.
\end{rem}

\begin{ex}
Consider the matrix (where $\lambda \in k\setminus 0$)
\[
X=\left(\begin{matrix}       0   & x_1 & x_2 & x_3  \\
       x_1 & x_2  & x_3 &0 \end{matrix}\right. \left. \dashline ~ \begin{matrix}
                                       y_1 & y_2   \\
                                       y_2 & y_3 
\end{matrix} \right. \left. \dashline ~\begin{matrix} z_1 & z_2 \\ 
                                                      z_2  &  0
                                     \end{matrix} \right.
                                                               \left. \dashline  ~\begin{matrix} u_1 & u_2 \\ 
                                                               u_2 + \lambda u_1 & \lambda u_2
                                                               \end{matrix}        \right).
\]
In this example, $c=d=1, t=2, g_1=g_2=1$. Moreover, $s_1=2$. Denote $H_{0,m_0-1}$ by $H_0$, $H_{1,r}$ by $H_r$, $I_{1,(a_1),(b_1)}$ by $I_{a_1,b_1}$, $J^{i,1,r}_{(a_1),(b_1)}$ by $J^{i,r}_{a_1,b_1}$ and $K^{\ell,i,1,r}_{(a_1),(b_1)}$ by $K^{\ell,i,r}_{a_1,b_1}$. The filtration \ref{constr_filtration} consists of the following ideals:
\begin{align*}
& H_0=(0), H_1=(x_2), H_2=(x_1,x_2), H_3=(x_1,x_2,x_3), I_{1,0}=(x_1,x_2,x_3,y_1), \\
& I_{2,0}=(x_1,x_2,x_3,y_1,y_2), I_{1,1}=(x_1,x_2,x_3,y_1,y_3), I_{2,1}=(x_1,x_2,x_3,y_1,y_2,y_3),\ldots,\\
& J^{1,1}_{1,0}=(x_1,x_2,x_3,y_1,z_1), J^{1,2}_{1,0}=(x_1,x_2,x_3,y_1,z_1,z_2), J^{1,1}_{2,1}=(x_1,x_2,x_3,y_1,y_2,y_3,z_1),\\
& J^{2,1}_{1,0}=(x_1,x_2,x_3,y_1,u_1), J^{2,2}_{2,1}=(x_1,x_2,x_3,y_1,y_2,y_3,u_1,u_2),\ldots,\\
& K^{2,1,1}_{1,0}=(x_1,x_2,x_3, y_1,u_1,u_2,z_1),K^{2,2,2}_{1,1}=(x_1,x_2,x_3,y_1,y_3,z_1,z_2,u_1,u_2),\\
& K^{2,2,2}_{2,0}=(x_1,x_2,x_3,y_1,y_2,z_1,z_2,u_1,u_2),\\
&\ldots,\\
& K^{2,2,2}_{2,1}=(x_1,x_2,x_3,y_1,y_2,y_3,z_1,z_2,u_1,u_2)=\mm.
\end{align*}
For example, we can check by Macaulay2 that:
\begin{enumerate}
\item $H_0:H_1=K^{2,2,2}_{1,1}, H_1:H_2=H_2:H_3=\mm$,
\item $H_3:I_{1,0}=H_3,$
\item $I_{1,0}:I_{2,0}=K^{2,2,2}_{2,0}$,
\item $I_{2,0}:I_{2,1}=K^{2,2,2}_{2,0}$,
\item $I_{2,1}:J^{1,1}_{2,1}=J^{2,2}_{2,1}$.
\end{enumerate}
\end{ex}

Let us now prove Theorem \ref{Koszul_filtration} by showing that Construction \ref{constr_filtration} really gives a Koszul filtration.

\begin{proof}[Proof of Theorem \ref{Koszul_filtration}] 
We show that the list of ideals 
$$\Fc = \{H_{i,j}\} ~ \bigcup ~ \{I_{s;\bsa,\bsb}\} ~ \bigcup ~ \{J_{\bsa,\bsb}^{i,j,r}\} ~ \bigcup ~ \{K_{\bsa,\bsb}^{\ell,i,j,r} \},$$
in Construction \ref{constr_filtration} gives a Koszul filtration for $R$.

From the definition of $\mathcal F$, the first two conditions of the definition of Koszul filtration follows immediately. For the colon condition, the following equalities hold.

\begin{enumerate}
\item For the ideal $H_{i,1}$ where $1 \le i\le c$, we have 
$$
H_{i-1,m_{i-1}-1}:H_{i,1}=H_{i-1,m_{i-1}-1}:x_{i,s_i} = K^{t,t,g_t,p_{tg_t}}_{\bsa_i,\bsb_i}
$$ 
where $\bsa_i$ and $\bsb_i$ are as in the Lemma \ref{H_series}(i): The left-hand side contains the right-hand side because of Lemma \ref{lem_identities}(ii) and Lemma \ref{H_series}(i). Working modulo $K^{t,t,g_t,p_{tg_t}}$ we may assume that $X$ has no Jordan blocks. The equality now follows from Lemma \ref{H_series}(i). 

\item For the ideal $H_{i,j}$ where $1 \le i \le c$ and $2 \le j \le m_i-1$, we have
 $$
 H_{i,j-1}:H_{i,j} = \mm.
 $$ 
 This follows from Lemma \ref{H_series}(ii) and Lemma \ref{lem_identities}(i).

\item For the ideal $I_{s,\bsa,\bsb}$ where $1\le s \le d$, and $\bsb$ is such that $b_i \ge 2$ for some $i$: denote $\widehat{\bsb} = (b_1, \ldots, b_{i-1}, b_i-1, b_{i+1}, \ldots, b_s)$. Then $I_{s;\bsa, \bsb} = I_{s;\bsa,\widehat{\bsb}} + (y_{i,n_i - b_i +2})$ and 
\[ I_{s;\bsa,\widehat{\bsb}}:y_{i,n_i-b_i+2} = \mm.\]
This follows from Lemma \ref{lem_identities}(iii) and Lemma \ref{I>=2_series}(i).
\item For the ideal $I_{s,\bsa,\bsb}$ where $1\le s \le d$, $b_1, \ldots, b_s \le 1$ and $\bsa$ is such that $a_i \ge 2$ for some $i$: denote $\widehat{\bsa} = (a_1, \ldots, a_{i-1}, a_i-1, a_{i+1}, \ldots, a_s)$. Then $I_{s;\bsa,\bsb} = I_{s;\widehat{\bsa},\bsb} + (y_{i,a_i})$ and 
\[
I_{s;\widehat{\bsa},\bsb}:y_{i,a_i}=\begin{cases}
\mm, &\textnormal{if $b_i= 1$};\\
K^{t,t,g_t,p_{tg_t}}_{(a'_1,\ldots,a'_{i-1},a'_i,\ldots,a'_d),0_d}, &\textnormal{if $b_i=0$},
\end{cases}
\]
where $a'_i=n_i$ and for $j\neq i$, $a'_j=\begin{cases}
n_j,  &\textnormal{if $n_j-a_j\ge n_i-a_i+1$},\\
n_j+1, &\textnormal{otherwise}.
\end{cases}$

This follows from Lemma \ref{I>=2_series}(ii) and Lemma \ref{lem_identities}(iii).

\item For the ideal $I_{s,\bsa,\bsb}$ where $1 \le s \le d$, $\bsa = 1_s$ and $b= 1_i$ for some $1 \le i \le s$: we have $I_{s;1_s;1_i} = I_{s;1_s,1_{i-1}} + (y_{i,n_i+1})$ and 
\[ 
I_{s;1_s,1_{i-1}}:y_{i,n_i+1} = K^{t,t,g_t,p_{tg_t}}_{(n_1+1,\ldots, n_{i-1}+1,1,n_{i+1}-n_i+1, \ldots, n_d-n_i+1),0_d}
\]
This follows from Lemma \ref{lem_identities}(iii) and Lemma \ref{I<=1_series}(i).
\item For the ideal $I_{s,\bsa,\bsb}$ where $1\le s \le d$, $\bsa = 1_s$ and $b = 0_s$: we have $I_{s,1_s,0_s} = I_{s-1,1_{s-1},0_{s-1}} + (y_{s,1})$ and 
\[ 
I_{s-1;1_{s-1},0_{s-1}}: y_{s,1} = I_{s-1, (n_1+1, \ldots, n_{s-1}+1),0_{s-1}}
\]
That the left-hand side contains the right-hand side follows from Lemma \ref{I<=1_series}(ii). Working modulo the right-hand side, we may assume that $X$ has no nilpotent blocks, and $s = 1$. We need to prove that $y_{1,1}$ is a non-zero divisor. This follows from Lemma \ref{lem_Groebner}.
\item Finally, for $J$ and $K$ series, the similar equalities hold true as in Lemma \ref{Jordan_filtration}.
\end{enumerate}
This completes the proof of Theorem \ref{Koszul_filtration}.
\end{proof}
\section{The necessary condition}
\label{Koszul_necessary}

For $m\ge 1$ and $n\ge 1$, consider the scroll of type $(m,n)$. It is given by the following matrix
\[
X=\left(\begin{matrix}
       x_1 & x_2  &\ldots  & x_{m}\\
       x_2 & x_3  &\ldots  & x_{m+1} \end{matrix}\right. \left. \dashline ~ \begin{matrix}
                                       y_1 & y_2 &\ldots &y_n \\
                                       y_2 & y_3 &\ldots &y_{n+1}
                                       \end{matrix} \right).
\]
\begin{thm}
\label{non-Koszul}
For any $n\ge 1$ and $m\ge 2n+1$, the ring $R(m,n)/(x_1,x_{m+1})$ is not Koszul.
\end{thm}
\begin{proof}
Denote $T=R(m,n)/(x_1,x_{m+1})$. We introduce some notations. Let $x,y,s_1$ and $s_2$ be  variables. Identify $\N^4$ with the multiplicative monoid $\left<x,y,s_1,s_2\right>$ by mapping a sequence of natural numbers $(g,h,p,q)$ to $x^gy^hs_1^ps_2^q$. Recall that $R(m,n)$ is the monoid ring $k[\Lambda]$ where $\Lambda$ is the following affine submonoid of $\N^4$
\[
\left<x^{m}s_1,x^{m-1}ys_1,\ldots,y^{m}s_1,x^{n}s_2,x^{n-1}ys_2,\ldots,y^ns_2\right>.
\]
Note that $R(m,n)$ is standard graded $k$-algebra by giving each of the minimal generators of $\Lambda$ the degree $1$.

Observe that $T$ has an induced $\Lambda$-grading and $k$ is a $\Lambda$-graded module. Denote $a=\lceil m/n \rceil$ and $\mu=x^{an}y^{m}s_1s_2^a$, an element of degree $a+1\ge 4$ of $\Lambda$.

{\bf Claim}: We always have $\beta_{3,\mu}^T(k)\ge 1$.

This implies that $\beta_{3,a+1}^T(k)\neq 0$, hence $T$ is not Koszul.

We will use a result of Herzog, Reiner and Welker \cite[Theorem 2.1]{HRW}, which gives the multigraded Betti numbers of $k$ over $T$. Denote by $\Delta_{\mu}$ the simplicial complex whose faces are sequences $\alpha_1 <\cdots <\alpha_s$ in $(0,\mu)$ where $\alpha_i\in \Lambda$. Let $J$ be the submonoid generated by $x^ms_1,y^ms_1$ of $\Lambda$. Note that $T=k[\Lambda]/(x^ms_1,y^ms_1)$. Denote by $\Delta_{\mu,J}$ the subcomplex of $\Delta_{\mu}$ consisting of sequences $\alpha_1 <\cdots <\alpha_s$ such that for some $0\le i\le s$, we have $\alpha_{i+1}/\alpha_i \in J$, where $\alpha_0=0$ and $\alpha_{s+1}=\mu$ by convention.

By \cite[Theorem 2.1]{HRW}, we have
\[
\beta_{3,\mu}(k)=\dim_k \widetilde{H}_1(\Delta_{\mu},\Delta_{\mu,J};k),
\]
where the left-hand side is the reduced, relative simplicial homology of the pair $\Delta_{\mu},\Delta_{\mu,J}$. There is an exact sequence
\[
\widetilde{H}_1(\Delta_{\mu};k) \to \widetilde{H}_1(\Delta_{\mu},\Delta_{\mu,J};k) \to \widetilde{H}_0(\Delta_{\mu,J};k) \to \widetilde{H}_0(\Delta_{\mu};k).
\]
Since $k[\Lambda]=R(m,n)$ is Koszul, by the same result cited above, the two terms on two sides of the above sequence are zero. Thus it is enough to show that $\widetilde{H}_0(\Delta_{\mu,J};k)\neq 0$, equivalently, $\Delta_{\mu,J}$ is disconnected.

There are two types of facets of $\Delta_{\mu,J}$: those sequences $\alpha_1 <\cdots <\alpha_s$ such that $\alpha_{i+1}/\alpha_i \in (y^ms_1)\Lambda$ for some $0\le i\le s$, and those such that $\alpha_{j+1}/\alpha_j \in (x^ms_1)\Lambda$ for some $0\le j\le s$. These two classes of facets are disjoint since $\mu$ is not a multiple of $s_1^2$ in $\N^4$.

Now $\mu=y^ms_1(x^ns_2)^a$, so if $\alpha_1 <\cdots <\alpha_s$ is a facet of the first type, the sequence
$(\alpha_{i+1}/\alpha_i)_{i=0}^s$ is (up to permutation) the sequence $(y^ms_1,x^ns_2,x^ns_2,x^ns_2,\ldots,x^ns_2)$ (there are $a$ elements $x^ns_2$). Therefore the only facets of the first type are of the form
\[
(x^ns_2,(x^ns_2)^2,\ldots,(x^ns_2)^t,y^ms_1(x^ns_2)^t,y^ms_1(x^ns_2)^{t+1},\ldots,y^ms_1(x^ns_2)^{a-1})\\
\]
for some $0\le t \le a$.

The following diagram illustrates the case $n=1, m=3$.

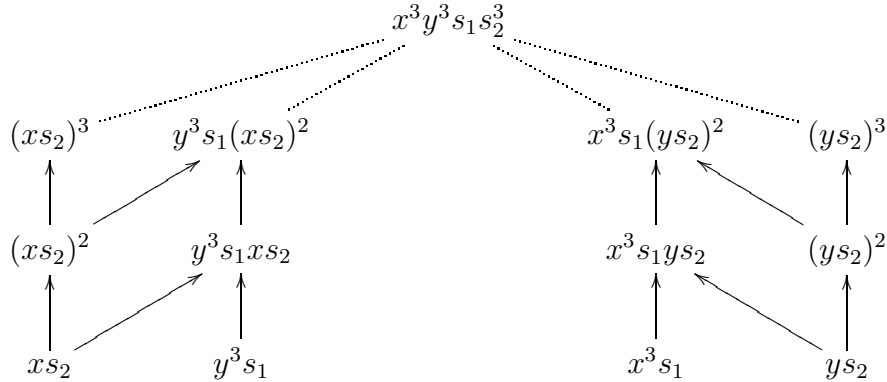
\begin{displaymath}
    \xymatrix{      &                    &              & x^3y^3s_1s_2^3                          &   &  &\\
   &(xs_2)^3 \ar@{.}[urr]  & y^3s_1(xs_2)^2 \ar@{.}[ur] &                 & x^3s_1(ys_2)^2 \ar@{.}[ul] &(ys_2)^3 \ar@{.}[ull] & &\\
 &(xs_2)^2  \ar[u] \ar[ur]    &  y^3s_1xs_2  \ar[u]                       &           & x^3s_1ys_2 \ar[u]   &(ys_2)^2 \ar[ul] \ar[u] &\\
 & xs_2 \ar[u] \ar[ur]        & y^3s_1 \ar[u]            &                &  x^3s_1 \ar[u]         & ys_2 \ar[u] \ar[ul] \\}
\end{displaymath}

In the diagram, the arrows signify divisibility of upper elements to the corresponding lower elements. The facets of $\Delta_{\mu,J}$ are maximal chains of arrows in the diagram.

Fix $1\le t\le a$. we show that no facet of second type may contains $(x^ns_2)^t$. Indeed, otherwise we have a facet $\alpha_1< \cdots <\alpha_s$ of second type where $\alpha_i=(x^ns_2)^t$ for some $1\le i\le s$. None of the quotient $\alpha_j/\alpha_{j-1}$ where $j\le i$ can be $x^ms_1$ since $x^ms_1$ is not a divisor of $(x^ns_2)^t$. Now $\alpha_{s+1}/\alpha_i= (\alpha_{s+1}/\alpha_s)\cdots (\alpha_{i+1}/\alpha_i)=x^{(a-t)n}y^ms_1s_2^{a-t}$. One of the quotients $\alpha_{j+1}/\alpha_j$ (where $j=i,i+1,\ldots,s$) is $x^ms_1$, hence the product of the remaining ones is $x^{(a-t)n-m}y^ms_2^{a-t}$. The last element does not belong to $\Lambda$ since $(a-t)n\le (a-1)n<m$, a contradiction.

Similarly, one can prove that no facet of second type may contains one of the elements $y^ms_1,y^ms_1xs_2,\ldots,y^ms_1(xs_2)^{a-1}$.

It is immediate that $(y^ns_2,y^{2n}s_2^2,\ldots,y^{(a-1)n}s_2^{a-1},x^{an-m}y^ms_2^a)$ is a facet of second type. Therefore $\Delta_{\mu,J}$ has at least $2$ connected components. (In fact, it has exactly $2$ components, as the interested reader can check that the facets of second type generate a connected complex.) Hence the claim is true, and the proposition is established.
\end{proof}
We are ready for the
\begin{proof}[Proof of the necessary condition in Theorem \ref{main}]
If $m\ge 2n+1$, the determinantal ring $A$ of the submatrix consisting of a nilpotent block of length $m$ and a scroll block of length $n$ is not Koszul, by Theorem \ref{non-Koszul}. Since $A$ is an algebra retract of $R$, from \cite[Proposition 1.4]{HHO}, $R$ is also not Koszul. This is a contradiction, hence $m\le 2n$.
\end{proof}
\begin{rem}
Let $R$ be a rational normal scroll and $Y$ a set of natural coordinates. Using Theorem \ref{main}, we can determine all $Y$ such that the quotient ring $R/(Y)$ is Koszul. Indeed, $R/(Y)$ is defined by a matrix consisting of scroll blocks with certain variables being replaced by zero. By the proof of Lemma \ref{KW_quotient}, one can find a Kronecker-Weierstrass normal form of $R/(Y)$ by first finding the normal form for each of these blocks. Such normal forms exist by Remark \ref{rem_quotient}. Then by Theorem \ref{main}, we easily determine whether $R/(Y)$ is Koszul or not.
\end{rem}


\section{Applications to linear sections of rational normal scrolls}
\label{sect_application}

We start this section by proving that all the linear sections of a scroll have a linear resolution if and only if the scroll is of type $(\nsf_1,\ldots,\nsf_1)$.
\begin{defn}
Let $R$ be a standard graded $k$-algebra with $r_1,\ldots,r_n$ being minimal homogeneous generators of $\mm$. We say that $R$ is {\em strongly Koszul} if for every sequence $1\le i_1 < i_2 <\cdots <i_s\le n$, the ideal $(r_{i_1},\ldots,r_{i_{s-1}}):r_{i_s}$ is an ideal generated by a subset of $\{r_1,\ldots,r_n\}$.
\end{defn}
\begin{rem}
Another notion of strongly Koszul algebras were introduced in \cite[Definition 3.1]{CNR}. The two notions are equivalent when $R=k[\Lambda]$ where $\Lambda$ is an affine monoid and $r_1,\ldots,r_n$ are the minimal generators of $\Lambda$.
\end{rem}
See \cite{HHR} for a detailed discussion of strongly Koszul algebras.
\begin{prop}
For a homogeneous affine monoid $\Lambda$ and $r_1,\ldots,r_n$ the minimal generators of $\Lambda$, the following are equivalent:
\begin{enumerate}
\item $R=k[\Lambda]$ is strongly Koszul;
\item $\reg_R R/(Y)=0$ for every subset $Y$ of $\{r_1,\ldots,r_n\}$.
\end{enumerate}
\end{prop}
\begin{proof}
That (i) implies (ii) is obvious: the ideals generated by subsets of $\{r_1,\ldots,r_n\}$ form a Koszul filtration for $k[\Lambda]$. Now assume that (ii) is true. For each subset $Y$ of $\{r_1,\ldots,r_n\}$ and $r_j\notin Y$, consider the short exact sequence
\[
0\to (Y)\cap (r_j)\to (Y)\oplus (r_j) \to (Y,r_j)\to 0.
\]
By the hypothesis, $\reg_R ((Y)\oplus (r_j))=\reg_R (Y,r_j)=1$. Hence $\reg_R ((Y)\cap (r_j))\le 2$. By \cite[Proposition 1.4]{HHR}, this implies that $R$ is strongly Koszul. 
\end{proof}
An immediate corollary is the following result due to Conca.
\begin{prop}[Conca \cite{Con1}]
\label{strongly_Koszul}
The scroll $R=R(\nsf_1,\ldots,\nsf_{\tsf})$ has the property that $\reg_R R/(Y)=0$ for every set of variables $Y$ if and only if $\nsf_1=\nsf_2=\cdots=\nsf_{\tsf}$.
\end{prop}
\begin{proof}
We prove that $R$ is strongly Koszul if and only if $\nsf_1=\nsf_2=\cdots=\nsf_{\tsf}$. The ``if" direction is clear: if $\nsf_1=\nsf_2=\cdots=\nsf_{\tsf}$ then $R$ is the Segre product of $k[s_1,\ldots,s_{\tsf}]$ and the $\nsf_1$th Veronese of $k[x,y]$. Therefore $R$ is strongly Koszul by \cite[Proposition 2.3]{HHR}.

The ``only if" direction: assume that the contrary is true, for example $\nsf_{\tsf}>\nsf_1$. Since $R$ is strongly Koszul, moding out the variables of the blocks of lengths $\tsf_2,\ldots,\tsf_{\nsf-1}$, we see that the scroll of type $(\nsf_1,\nsf_{\tsf})$ is also strongly Koszul. We will deduce a contradiction. For simplicity we can assume that $\tsf=2$.

Denote $a=\nsf_1,b=\nsf_2$. Let $r=\lceil b/a\rceil$. The ring $R$ is also an affine monoid ring, $R\cong k[x^as_1,x^{a-1}ys_1,\ldots,y^as_1,x^bs_2,x^{b-1}ys_2,\ldots,y^bs_2] \subseteq k[x,y,s_1,s_2]$ where $x,y,s_1,s_2$ are variables. From \cite[Proposition 1.4]{HHR}, the ideal $(x^as_1):y^as_1$ is generated by a subset of 
\[
\{x^as_1,x^{a-1}ys_1, \ldots,y^as_1,x^bs_2,\ldots,y^bs_2\}.
\] 
However, $x^by^{ra-b}s_2^r$ is clearly a minimal generator of $(x^as_1):y^as_1$ and it does not belong to the above-mentioned set, since it has degree $r\ge 2$ in $R$. This is a contradiction.
\end{proof}

\begin{defn}
Let $R$ be a standard graded $k$-algebra with graded maximal ideal $\mm$. Let $R=S/I$ be a presentation of $R$ where $S=k[\xsf_1,\ldots,\xsf_n]$ be a standard graded polynomial ring and $I$ a homogeneous ideal of $S$. The algebra $R$ is a called
{\em linearly Koszul} (with respect to the sequence $\overline{\xsf_1},\ldots,\overline{\xsf_n}$) if $R/(Y)$ is a Koszul algebra for every subsequence $Y$ of $\boldsymbol{\xsf}=\overline{\xsf_1},\ldots,\overline{\xsf_n}$.

We say that $R$ {\em satisfies the regularity condition} if $\reg R/(Y) \le \reg R$ for every subsequence $Y$ of $\boldsymbol{\xsf}$, where $\reg$ denotes the absolute Castelnuovo-Mumford regularity.
\end{defn}
\begin{rem}
(i) Any algebra defined by quadratic monomial relations is Koszul by the result of Fr\"oberg \cite{Fr1}, and consequently it is also linearly Koszul.

(ii) If $R$ is linearly Koszul then so is quotient ring $R/(Y)$ for every subsequence $Y$ of $\boldsymbol{\xsf}$.

(iii) If $R$ is strongly Koszul with respect to the sequence $\boldsymbol{\xsf}$ then it is also linearly Koszul. The reverse implication is not true, even if $R$ is defined by all monomial relations except one binomial relation. For example, let $R$ be the determinantal ring of the following matrix
\[
\left(\begin{matrix} x   & 0 & z\\
                     y   & z & t
                         \end{matrix}\right)
\]
Concretely $R=k[x,y,z,t]/(xz,z^2,xt-yz)$. Then $y$ is an $R$-regular element and $R/(y)\cong k[x,z,t]/(xz,z^2,xt)$ is Koszul, so $R$ is also Koszul, e.g.~ by Lemma \ref{lem_regularity}. It is also easy to check that each of the quotient rings $R/(x),R/(z),R/(t)$ is a Koszul algebra defined by monomial relations. Therefore $R$ is linearly Koszul. On the other hand $0:t=(x^2)$, hence $R$ is not strongly Koszul (with respect to the natural coordinates).
\end{rem}
Note that if $R$ is a rational normal scroll, $\reg R=1$. In this case, we have: 
\begin{lem}
If $\reg(R)=1$ and $R$ satisfies the regularity condition then $R$ is linearly Koszul.
\end{lem}
\begin{proof}
Take any standard graded polynomial ring $S$ which surjects onto $R$. From $\reg_S R=1$, we get $\reg_R k \le \reg_S k=0$ by Lemma \ref{lem_regularity}.

Denote by $\boldsymbol{\xsf}$ the sequence of natural coordinates of $R$. For every subsequence $Y$ of $\boldsymbol{\xsf}$, we have $\reg_R R/(Y) \le \reg R/(Y)\le 1$. By Lemma \ref{lem_regularity}, this implies $\reg_{R/(Y)}k \le \reg_R k=0$. Hence $R/(Y)$ is Koszul.
\end{proof}

We are ready for Theorem \ref{linearly_Koszul}(i) which characterizes balanced scrolls in terms of the regularity condition. That this could be true was predicted by Conca \cite{Con1}.
\begin{thm}
\label{reg_Koszul}
A rational normal scroll satisfies the regularity condition if and only if it is balanced.
\end{thm}
\begin{proof}
Assume that the scroll is balanced $R=R(\nsf_1,\ldots,\nsf_1,\nsf_1+1,\ldots,\nsf_1+1)$. For every set of variables $Y$, the quotient ring $R/(Y)$ is the determinantal ring of a $2\times e$ matrix $X$ of linear forms, which can assumed to be in Kronecker-Weierstrass form. By Proposition \ref{KW_form}, the length of any scroll block of $X$ (if exists) is at least $\nsf_1$. By Lemma \ref{KW_quotient}, each nilpotent block of $X$ has length at most $\nsf_1+1$. Therefore $\reg R/(Y) \le 1$ by Theorem \ref{regularity_of_generalized_scroll}, as desired.

The necessary condition is immediate from Theorem \ref{regularity_of_generalized_scroll}. In our case,
\[
\reg R(\nsf_1,\ldots,\nsf_{\tsf})/(x_{\tsf,1},x_{\tsf,\nsf_{\tsf}+1})=\left \lceil \frac{\nsf_{\tsf}-1}{\nsf_1} \right \rceil  \ge 2
\]
if $\nsf_{\tsf}\ge \nsf_1+2$.
\end{proof}
Now we prove Theorem \ref{linearly_Koszul}(ii) which characterizes linearly Koszul scrolls.
\begin{thm} The scroll $R=R(\nsf_1,\ldots,\nsf_{\tsf})$ is linearly Koszul if and only if $\nsf_{\tsf} \le 2\nsf_1$.
\end{thm}
\begin{proof} For the sufficient condition: assume that $\nsf_{\tsf}\le 2\nsf_1$. Take any set of natural coordinates $Y$. Let $X$ be the matrix of linear forms defining $R/(Y)$. From Proposition \ref{KW_form} and Lemma \ref{KW_quotient}, any canonical form of $X$ satisfies the length condition. By Theorem \ref{Koszul_filtration}, we conclude that $R/(Y)$ is Koszul. 

The necessary condition follows from Theorem \ref{non-Koszul}.
\end{proof}

Next we consider the following class of linearly Koszul algebras, first introduced in \cite{Con1} under a different name.
\begin{defn}
Let $R$ be a standard graded $k$-algebra. We say that $R$ is {\em universally linearly Koszul} (abbreviated ul-Koszul) if $R/(Y)$ is a Koszul ring for every set of linear forms $Y$.
\end{defn}
\begin{rem}
We know that every Koszul algebra defined by quadratic monomial relations are linearly Koszul. However a Koszul algebra defined by quadratic monomial relations need not be universally linearly Koszul. Indeed, let 
\[
R=k[x,y,z,t,u,v]/(x^2,xy,y^2,xz,yt,uv) 
\]
and $I=(x+y-u,z-t-v)$. Then $R/I\cong k[x,y,z,t]/(x^2,xy,y^2,xz,yt,xt-yz)$ is not Koszul: it is defined by the matrix
\[
\left(\begin{matrix}0 & x   & y & z\\
                    x & y   & 0 & t
                         \end{matrix}\right)
\]
and by Theorem \ref{non-Koszul}, $R/I$ is not Koszul.
\end{rem}
In \cite{Con2}, the author defines $R$ to be universally Koszul if $\reg_R R/(Y)=0$ for every sequence of linear forms $Y$. Clearly every universally Koszul algebra is ul-Koszul. In the same paper, the universally Koszul rational normal scrolls of type $(\nsf_1,\ldots,\nsf_{\tsf})$ are completely classified: either $\tsf=1$ (a rational normal curve) or $\tsf=2$ and $\nsf_1=\nsf_2$. Using the classification of the Kronecker-Weierstrass normal forms of linear sections of rational normal scrolls in Section \ref{background}, we prove:
\begin{thm}
\label{ul_Koszul_scrolls}
The rational normal scroll $R(\nsf_1,\ldots,\nsf_{\tsf})$ is ul-Koszul if and only if either $\tsf=1$, or $\tsf= 2$ and $\nsf_2\le 2\nsf_1$, or $\tsf=3$ and $\nsf_1=\nsf_2=\nsf_3$.
\end{thm}
\begin{proof}
If the necessary condition is not true, then $\nsf_2+\cdots+\nsf_{\tsf}\ge 2\nsf_1+1$. Moding out a suitable sequence of binomial linear forms $Y$, we arrive at the ring $R(\nsf_1,\nsf_2+\cdots+\nsf_{\tsf})$. By Theorem \ref{non-Koszul} we get $R/(Y)$ is not linearly Koszul. Hence $R$ is not ul-Koszul.

The converse follows from Theorem \ref{main} and Proposition \ref{KW_form}: for any quotient ring by a linear ideal of $R$, any of its corresponding Kronecker-Weierstrass matrices satisfies the length condition.
\end{proof}
Conca \cite{Con3} discovered the classification of universally Koszul algebras defined by monomial relations. It would be interesting to classify all universally linearly Koszul algebras defined by monomial relations. 

Finally, similarly to Theorem \ref{ul_Koszul_scrolls}, we can classify scrolls that satisfy the ``universal" version of the regularity condition.
\begin{thm}
\label{univer_reg_scrolls}
The rational normal scroll $R=R(\nsf_1,\ldots,\nsf_{\tsf})$ has the property that $\reg R/(Y) \le \reg R$ for any set of linear forms $Y$ if and only if $\tsf\le 1$, or $\tsf=2$ and $\nsf_2\le \nsf_1+1$, or $\tsf=3$ and $\nsf_1=\nsf_2=\nsf_3=1$.
\end{thm}
\begin{proof}
If the necessary condition is not true, then $\nsf_2+\cdots+\nsf_{\tsf}\ge \nsf_1+2$. Moding out suitable linear forms, we arrive at the determinantal of a scroll block of length $\nsf_1$ and a nilpotent block of length $\nsf_2+\cdots+\nsf_{\tsf}$. The regularity of that ring is at least $2$ by Theorem \ref{regularity_of_generalized_scroll}. This is a contradiction.

For the sufficient condition: one only has to use Proposition \ref{KW_form} and Theorem \ref{regularity_of_generalized_scroll}.
\end{proof}
\section*{Acknowledgments}
We are grateful to Aldo Conca for his suggestion of the problems and stimulating discussions. We would like to thank the referee for several useful advice that helped us to correct errors from the previous version and streamline the presentation.

\end{document}